\documentclass{amsart}
\usepackage[dvipsnames]{xcolor}
\usepackage{amsmath,amsfonts,amsthm,amssymb,mathrsfs,latexsym,paralist,comment, bbold, epsfig, bbm}
\usepackage[normalem]{ulem} 
\usepackage{tikz,pgfpages}
\usepackage{cancel} 
\usepackage{natbib,hyperref}

\newcommand{\CC}{\mathbb{C}}

\newcommand{\NN}{\mathbb{N}}
\newcommand{\ZZ}{\mathbb{Z}}

\makeatletter
\newcommand{\myitem}[1]{%
\item[#1]\protected@edef\@currentlabel{#1}%
}
\makeatother 

\newtheorem{thm}{Theorem}[section]

\newtheorem{cor}[thm]{Corollary}
\newtheorem{lemma}[thm]{Lemma}

\theoremstyle{remark}
\newtheorem{rmk}[thm]{Remark}

\newtheorem{example}[thm]{Example}
\newtheorem*{example*}{Example}

\theoremstyle{definition}
\newtheorem{defn}{Definition}

	\newtheorem{lem}[thm]{Lemma}

	\theoremstyle{definition}

	\theoremstyle{remark}

    \newcommand{\Kk}{\mathcal{K}}

\title{A New Geometric Morita Invariant for Higher Rank Graph $C^*$-algebras}

\author[M. Amann]{Mackenzie Amann}
\address{Department of Mathematics, University of Nebraska-Lincoln, 203 Avery Hall,
P.O. Box 880130,
Lincoln, NE 68588-0130 USA}
\email{mamann2@huskers.unl.edu}

\author[L. Gallagher]{Liam Gallagher}
\address{Department of Mathematics, Lehigh University, Chandler-Ullmann Hall, 17 Memorial Drive East,
Bethlehem, PA 18015 USA}
\email{lig225@lehigh.edu}

\author[R. Norton]{Rachael Norton}
    \address{Department of Mathematics, Statistics, and Computer Science, St. Olaf College, 1520 St. Olaf Ave, Northfield, MN 55057 USA}
    \email{norton10@stolaf.edu}

    \author[E. Ruiz]{Efren Ruiz}
    \address{Department of Mathematics, University of Hawaii, Hilo, 200 W. Kawili St., Hilo,
Hawaii, 96720-4091 USA
}
    \email{ruize@hawaii.edu}
\date{\today}

\begin{document}
\maketitle

\begin{abstract}
    Higher rank graphs, also known as $k$-graphs, are a $k$-dimensional generalization of directed graphs and a rich source of examples of $C^*$-algebras and $\ZZ^k$-graded rings. In the present paper, we contribute to the geometric classification program for $k$-graph algebras by introducing a new move on $k$-graphs, called LiMaR-split, which is a generalization of outsplit for directed graphs. We show that LiMaR-split preserves the $k$-graph $C^*$-algebras up to equivariant Morita equivalence and preserves the Kumjian-Pask algebras up to graded Morita equivalence.  
\end{abstract}

\section{Introduction}
This paper contributes to the geometric classification program of $k$-graph algebras, which was begun by Eckhardt et al.~in \cite{efgggp} and extends the geometric approach to classifying directed graph $C^*$-algebras. Geometric classification focuses on describing a set of changes that can be made to a graph, called \textit{graph moves}, that preserve the Morita equivalence class of the associated $C^*$-algebra. It provides an alternative, and complementary, approach to classification by $K$-theoretic invariants (e.g. Elliott's classification program (\cite{TWW17}, \cite{EGLN15}, \cite{GLN21b},\cite{GLN21a})), and its origins trace back to the 1970's and 80's. In \cite{PS75}, Parry and Sullivan showed that flow-equivalence of non-negative integral matrices is generated by two moves on the matrices. Since non-negative integral matrices are used in symbolic dynamics to construct subshifts of finite type, Franks used Parry and Sullivan's result to classify irreducible subshifts of finite type up to flow-equivalence in \cite{Fra84}. Cuntz and Krieger then associated to any non-negative integral matrix with no zero rows or columns a $C^*$-algebra $\mathcal{O}_A$ and showed that if $B$ and $C$ are flow-equivalent irreducible matrices, then $\mathcal{O}_B$ and $\mathcal{O}_C$ are Morita equivalent (\cite{CK80}). Viewing $\mathcal{O}_B$ and $\mathcal{O}_C$ as graph algebras with adjacency matrices $B$ and $C$, respectively, the matrix moves that implement flow-equivalence of $B$ and $C$ can be interpreted as graph moves.

It is this line of thinking that S{\o}rensen credits in \cite{Sorensen}, where he describes five graph moves -- source deletion \texttt{(S)}, insplit \texttt{(I)}, outsplit \texttt{(O)}, reduction \texttt{(R)}, and Cuntz splice \texttt{(C)} -- and proves that \texttt{(S)}, \texttt{(I)}, \texttt{(O)}, and \texttt{(R)} preserve the Morita equivalence class of the associated graph $C^*$-algebra, assuming it is simple and unital. That \texttt{(C)} also preserves Morita equivalence is proven in \cite{ERRS-CuntzSplice}. Then, Eilers et al.~in \cite{ERRS21} define a sixth move -- enclosing a cyclic component \texttt{(P)} -- on directed graphs with finitely many vertices that preserves the Morita equivalence class of the associated (unital) $C^*$-algebra. Moreover, they prove that any two unital graph $C^*$-algebras $C^*(E)$ and $C^*(F)$ are Morita equivalent if and only if $E$ can be obtained from $F$ by performing a finite sequence of these six moves and their inverses. This completes the geometric classification of unital graph $C^*$-algebras, which importantly includes all Cuntz-Krieger algebras. 

A next natural step is to seek similar results for higher rank graph $C^*$-algebras. Introduced by Kumjian and Pask in \cite{KP}, higher rank graphs, also known as $k$-graphs, are a generalization of directed graphs to $k$ dimensions and provide a richer class of examples of $C^*$-algebras. Intuitively, a $k$-graph is a directed graph with edges in $k$ colors satisfying a so-called \textit{factorization property} (see Definition \ref{def: k-graph} for details). In fact, in \cite{Hazlewood} Hazlewood et al.~formalize the relationship between $k$-graphs and graphs with edges in $k$ colors (also called $k$-colored graphs), a relationship we rely on heavily in this paper.

To classify $k$-graph $C^*$-algebras geometrically, we seek to identify a finite set of moves that can be used to pass between any two $k$-graphs whose $C^*$-algebras are Morita equivalent. Eckhardt et al.~initiate the geometric classification program for $C^*$-algebras of row-finite, source-free $k$-graphs in \cite{efgggp}. They generalize four moves (insplitting, delay, sink deletion, and reduction) used in classifying directed graph $C^*$-algebras to the setting of $k$-graphs and prove that, under certain assumptions, each move yields a $k$-graph whose $C^*$-algebra is Morita equivalent to the original $C^*$-algebra; in the case of insplitting, the $C^*$-algebras are actually isomorphic. Their work lays the foundation for Lippert's definition of reduction in \cite{lippert2024} and a generalization of insplitting described by Hargreaves, Nemeth, and Oberbillig in \cite{hargreaves}. Lippert weakens the hypotheses assumed by Eckhardt et al., extending their definition of reduction to a move that acts as a true inverse of delay and is robust under products. Importantly, Lippert introduces a notion of neighborhood on which reduction takes place, making it a local rather than global move. Similarly, Hargreaves et al.~define an alternative to insplitting, called IKI-splitting, which acts on a neighborhood. A benefit of IKI-split is that it does not rely on the pairing condition required for the insplit defined by Eckhardt et al. Similarly to Eckhardt et al., in \cite{hargreaves} Hargreaves et al.~show that the $C^*$-algebra of the IKI-split is isomorphic to the $C^*$-algebra of the original $k$-graph. Lastly, Listhartke defines outsplit for $k$-graphs at a single vertex in \cite{listhartke}. Like the insplit defined by Eckhardt et al., Listhartke's outsplit relies on a pairing condition, which restricts the applicability of the move.

In this paper, we present a new definition of outsplit called LiMaR-split. Inspired by the success of Hargreaves et al., the aim of this project was to describe an outsplit, defined on a neighborhood, which did not rely on a pairing condition. Indeed, LiMaR-split acts on a neighborhood, not at a single vertex like Listhartke's outsplit.  We show in Corollary~\ref{cor:stb} that LiMaR-split preserves the Morita equivalence class of the associated $C^*$-algebras.  In fact, we prove something stronger; it is shown that the $C^*$-algebras associated to a $k$-graph and its LiMaR-split are stably $*$-isomorphic via a $*$-isomorphism that preserves the diagonal subalgebras and commutes with the $k$-torus actions on the $C^*$-algebras.  An immediate consequence of our main result is that the $\NN^k$-actions associated to the $k$-graph and its LiMaR-split are conjugate (Corollary~\ref{cor:dynamics}).

The proof of Corollary~\ref{cor:stb} is purely algebraic, where we mainly work with Kumjian-Pask algebras.  These algebras were introduced by Aranda Pino, Clark, an Huef, and Raeburn in 2013 in \cite{KPAlgebras} and are an algebraic analogue to $k$-graph $C^*$-algebras in the same way that Leavitt path algebras are an algebraic analogue of directed graph $C^*$-algebras.  For a $k$-graph $\Lambda$ and its LiMaR-split $\Gamma$, our main algebraic result (Theorem~\ref{thm-KPA-new}) shows that the Kumjian-Pask algebra $\mathrm{KP}_R(\Lambda)$ is $\ZZ^k$-graded $*$-isomorphic to a corner of $\mathrm{KP}_R(\Gamma)$, and the $\ZZ^k$-graded $*$-isomorphism sends the diagonal to a corner of the diagonal.  We then prove our main $C^*$-algebraic result by exploiting the fact that the Kumjian-Pask algebra over $\CC$ of a $k$-graph is a dense subalgebra of the $k$-graph $C^*$-algebra.

The structure of the paper is as follows. In Section \ref{sec: preliminaries} we recall pertinent definitions and theorems about $k$-graphs and their $C^*$-algebras. In particular, we state Theorem \ref{KGs}, which is due to Hazlewood et al.~in \cite{Hazlewood} and explains how to view $k$-graphs as $k$-colored directed graphs that satisfy certain additional conditions. We also review relevant definitions and theorems about Kumjian-Pask algebras.  In Section \ref{sec: limar-split} we define the LiMaR-split of a source-free, row-finite $k$-graph and prove that, under the mild assumption of sink-free, the resulting graph is indeed a $k$-graph (Theorem \ref{thm:kgraph}). Here we rely heavily on the perspective of $k$-graphs as $k$-colored graphs that satisfy Theorem \ref{KGs}. The goal of the final section, Section \ref{sec: ME}, is to investigate the relationship between $C^*(\Lambda)$ and $C^*(\Gamma)$ as well as the relationship between the Kumjian-Pask algebra over $\Lambda$ and the Kumjian-Pask algebra over $\Gamma$, where $\Gamma$ is the LiMaR-split of $\Lambda$.  We proceed to prove a few technical lemmas which hold in the Kumjian-Pask algebra setting as well as the $C^*$-algebra setting, although we only state the results for Kumjian-Pask algebras.  These technical lemmas are used to prove our main results, Theorem~\ref{thm-KPA-new} in the algebraic setting and Corollary~\ref{cor:stb} in the $C^*$-algebraic setting.

\section{Preliminaries} \label{sec: preliminaries}
\subsection{Higher rank graphs}
We let $\NN=\{0,1,2,\ldots\}$ denote the monoid of natural numbers under addition and $\NN^+ = \{1,2,\ldots\}$ denote the natural numbers excluding $0$. For $k \in \NN^+$, we view $\NN^k$ as a category where composition of morphisms is given by addition. We let $\{\mathbf{e}_i\}_{1 \leq i \leq k}$ denote the standard basis of $\NN^k$. For $\mathbf{n} = \sum_{i=1}^kn_i\mathbf{e}_i \in \NN^k$, we write $|\mathbf{n}| = \sum_{i=1}^k n_i$.

The objects of study in this paper are higher rank graphs, also known as $k$-graphs. These objects are generalizations of directed graphs, with $k$ representing the number of different colors edges can have. When we draw $2$-graphs, we will use red (solid) edges and blue (dashed) edges. We now present the formal definition of $k$-graphs, which first appeared in \cite{KP}.

\begin{defn} \label{def: k-graph} \cite[Definition 1.1]{KP}
Fix $k \in \mathbb{N}^+$. Let $\Lambda$ be a countable small category and $d: \Lambda \to \NN^k$ be a functor. If $(\Lambda, d)$ satisfies the \textit{factorization rule}, that is, if for every $\lambda \in \Lambda$ and $\mathbf{m},\mathbf{n} \in \NN^k$ such that $d(\lambda) = \mathbf{m} + \mathbf{n}$, there are unique elements $\mu, \nu \in \Lambda$ such that $\lambda = \mu\nu$, $d(\mu) = \mathbf{n}$, and $d(\nu) = \mathbf{m}$, then $(\Lambda, d)$ is a \textit{k-graph}.
\end{defn}

For a $k$-graph $\Lambda$, we write $\Lambda^1 = \{\lambda \in \Lambda: |d(\lambda)| = 1\}$ and $\Lambda^0 = d^{-1}(\mathbf{0})$. If $e$ is in $\Lambda^1$, we say $e$ is an \textit{edge} in $\Lambda$. For $v \in \Lambda^0$, we say $v$ is a \textit{vertex}. We  define $\Lambda^*$ to be the set of all finite paths, that is all finite sequences of edges in $\Lambda$, and we define the length of a path $\lambda$, denoted $|\lambda|$ to be the number of edges in the path. That is, $|\lambda| = |d(\lambda)|$.

Because of factorization rules, for each $\lambda \in \Lambda$, we are guaranteed the existence of unique $v, w \in \Lambda^0$ such that $v\lambda w = \lambda$. In this case, we will write $r(\lambda) = v$ and $s(\lambda) = w$. For any $v \in \Lambda^0$ and $\mathbf{n} \in \NN^k$, we write
\begin{align*}
    & v\Lambda = \{\lambda \in \Lambda : r(\lambda) = v\} \\
    & v\Lambda^{\mathbf{n}} = \{\lambda \in v\Lambda : d(\lambda) = \mathbf{n}\}.
\end{align*}
The sets $\Lambda v$ and $\Lambda^{\mathbf{n}} v$ are defined analogously. We will occasionally use the notation $r^{-1}_{\mathbf{n}}(v)$ in place of $v\Lambda^{\mathbf{n}}$ and $s^{-1}_{\mathbf{n}}(v)$ in place of $\Lambda^{\mathbf{n}} v$.
Let $\mathbb{1} = \sum_{i=1}^k \mathbf{e}_i$. Then $\Lambda^{\mathbb{1}}$ denotes the set of rainbow paths of length $k$ in $\Lambda$. 

\begin{defn}\label{rowfinite}
A $k$-graph $\Lambda$ is \textit{row-finite} if for every $v \in \Lambda^0$ and each $\mathbf{e}_i \in \NN^k$ the set $v\Lambda^{\mathbf{e}_i}$ is finite.
\end{defn}

\begin{defn}\label{sourcefree}
Let $\Lambda$ be a $k$-graph. If $v\Lambda^{\mathbf{e}_i} \neq \emptyset$ for all $v \in \Lambda^0$ and every $\mathbf{e}_i \in \NN^k$, then $\Lambda$ is \textit{source-free}. 
\end{defn}
 
\begin{defn} \label{defn: d(B) sink}
    Let $\Lambda$ be a $k$-graph. We say $v \in \Lambda^0$ is a \textit{degree $\mathbf{e}_i$ sink} if $\Lambda^{\mathbf{e}_i}v=\emptyset$ for some $1 \leq i \leq k$. We say $\Lambda$ is \textit{degree $\mathbf{e}_i$ sink-free} if it has no degree $\mathbf{e}_i$ sinks.
\end{defn}

\textit{In this paper, we will assume that all $k$-graphs are row-finite and source-free.} Any additional assumptions will be made explicit in our theorem statements.

Similarly to \cite{Hazlewood}, we model $k$-graphs using $k$-colored graphs called \textit{$1$-skeletons} by letting $G = (G^0, G^1, r, s)$ denote a directed graph where $G^0$ is its set of vertices, $G^1$ its set of edges, and $r,s: G^1 \to G^0$ are the range and source maps, respectively. For an integer $n \geq 2$, we let $G^n$ denote the paths of length $n$ in $G$. 

Then, we color the graph $G$ by assigning each edge one of the standard basis vectors $\mathbf{e}_i$ of $\NN^k$, letting $G^{\mathbf{e}_i}$ be the set of edges assigned to $\mathbf{e}_i$, such that $G^1 = \bigcup_{i=1}^k G^{\mathbf{e}_i}$. Then, we can assign to the path category $G^* = \bigcup_{n \in \NN} G^n$ a degree functor $d: G^* \to \NN^k$, where $d(v) = \mathbf{0}$ for all $v \in G^0$. For edges $f \in G^1$, $d(f) = \mathbf{e}_i$ if $f$ was assigned the basis vector $\mathbf{e}_i$. On paths, $d$ is extended to be additive, that is $d(f_n\dotsb f_1) = \sum_{i=1}^n d(f_i)$. Additionally, the range and source maps $r,s: G^1 \to G^0$ extend to well-defined maps from $G^* \to G^0$ and will also be denoted by $r$ and $s$.

The following theorem makes explicit the relationship between $k$-colored directed graphs and $k$-graphs.

\begin{thm} \label{KGs} \cite[Theorems 4.4 and 4.5]{Hazlewood}\cite[Section 2]{efgggp} 

If $G$ is a $k$-colored directed graph equipped with a degree map $d: G^* \to \NN^k$ and range and source maps $r,s: G^1 \to G^0$, then $(\Lambda = G^* /\sim,d)$ is a $k$-graph for any $(r,s,d)$-preserving equivalence relation $\sim$ on $G^*$ which also satisfies the following four conditions:

\begin{enumerate}
\myitem{(KG0)} \label{KG0} If $\lambda \in G^*$ is a path such that $\lambda = \lambda_2\lambda_1$, then $[\lambda] = [p_2p_1]$ whenever $p_1 \in [\lambda_1]$ and $p_2 \in [\lambda_2]$.

\myitem{(KG1)} \label{KG1} If $f,g \in G^1$ are edges, then $f \sim g$ if and only if $f = g$.

\myitem{(KG2)} \label{KG2} \textbf{Completeness:} For every $\mu = \mu_2\mu_1 \in G^2$ such that $d(\mu_1) = \mathbf{e}_i$, $d(\mu_2) = \mathbf{e}_j$, there exists a unique $\upsilon = \upsilon_2\upsilon_1 \in G^2$ such that $d(\upsilon_1) = \mathbf{e}_j$, $d(\upsilon_2) = \mathbf{e}_i$, and $\mu \sim \upsilon$.

\myitem{(KG3)} \label{KG3} \textbf{Associativity:} For any $\mathbf{e}_i-\mathbf{e}_j-\mathbf{e}_{\ell}$ path $abc \in G^3$ with $i, j, \ell$ all distinct, the $\mathbf{e}_{\ell}-\mathbf{e}_j-\mathbf{e}_i$ paths $hjg, nrq$ constructed via the following two routes are equal.
\end{enumerate}

\begin{minipage}[r]{0.35\textwidth}
\begin{align*}
\text{Route 1: } & \text{Let } ab \sim de, \text{so } abc \sim dec \\
& \text{Let } ec \sim fg, \text{so } abc \sim dfg\\
& \text{Let } df \sim hj, \text{so } abc \sim hjg\\
\end{align*}
\end{minipage}
\quad 
\begin{minipage}[r]{0.35\textwidth}
\begin{align*}
\text{Route 2: } & \text{Let } bc \sim km, \text{so } abc \sim akm \\
& \text{Let } ak \sim np, \text{so } abc \sim npm\\
& \text{Let } pm \sim rq, \text{so } abc \sim nrq\\
\end{align*}
\end{minipage}

Note that if $\lambda\sim\mu$ in a $k$-graph, we say that $\lambda$ \emph{commutes} with $\mu$. 
\end{thm}

In order to define LiMaR-split on $k$-graphs, we will use this association between $1$-skeletons with equivalence relations and $k$-graphs. More precisely, we will define LiMaR-split on a $k$-graph $\Lambda$ by describing how it impacts the $1$-skeleton $G$ and the equivalence relation $\sim$ associated with $\Lambda$. This produces a new $1$-skeleton and equivalence relation, denoted $G_L$ and $\sim_{L}$ respectively. Then, in order to show $G_L^* / \sim_{L}$ is a $k$-graph, we show that $\sim_{L}$ satisfies \ref{KG0}-\ref{KG3}.

\subsection{Higher rank graph $C^*$-algebras}
In this subsection we introduce key definitions related to higher rank graph $C^*$-algebras. For a more complete review of higher rank graph $C^*$-algebras, we refer the reader to \cite[Section 1]{KP}. 

\begin{defn}\label{def: CKs} \cite[Definition 1.5]{KP}\cite[Definition 7.4]{KPS}
Let $\Lambda$ be a row-finite, source-free k-graph. A \textit{Cuntz-Krieger $\Lambda$-family} is a collection of projections $\{P_v: v \in \Lambda^0 \}$ and partial isometries $\{T_e: e \in \Lambda^1\}$ satisfying the \textit{Cuntz-Krieger relations}:

\begin{enumerate}
\myitem{(CK1)} \label{CK1} The projections $P_v$ are mutually orthogonal.

\myitem{(CK2)} \label{CK2} If $a,b,f,g \in \Lambda^1$ satisfy $af \sim gb$, then $T_aT_f = T_gT_b$.

\myitem{(CK3)} \label{CK3} For any $f \in \Lambda^1$, we have $T_f^*T_f = P_{s(f)}$.

\myitem{(CK4)} \label{CK4} For any $v \in \Lambda^0$ and any $1 \leq i \leq k$, we have $P_v = \sum\limits_{f\in v\Lambda^{\mathbf{e}^i}}T_fT_f^*$. 
\end{enumerate}

These generators and relations give rise to a universal $C^*$-algebra denoted $C^*(\Lambda) = C^*(\{p_v,t_f\})$. Whenever possible, we will use lowercase letters to denote the universal Cuntz-Krieger $\Lambda$-family $\{p_v, t_f\}$. For any Cuntz-Krieger $\Lambda$-family $\{P_v, T_f\}$, we then have a surjective $*$-homomorphism $\pi : C^*(\Lambda) \to C^*(\{P_v, T_f\})$ such that $\pi(p_v) = P_v$ and $\pi(t_f) = T_f$ for all $v \in \Lambda^0$, $f \in \Lambda^1$.
\end{defn}

\begin{rmk} \label{rmk: completeCK4} The relation \ref{CK4} is sometimes stated for paths instead of edges as follows: For any $v \in \Lambda^0$ and for any $\mathbf{n} \in \mathbb{N}^k$, we have $P_v = \sum_{f \in v\Lambda^{\mathbf{n}}} T_f T_f^*$. An induction proof shows that this statement is equivalent to our \ref{CK4} (c.f. \cite[Lemma 3.2]{RS2}).
\end{rmk}

\begin{lemma}\label{lem: consequences of CKs} \cite[Remark 2.2]{efgggp}
If $\{P_v, T_f\}$ is a Cuntz-Krieger $\Lambda$-family, then  $T_fP_{s(f)} = T_f$ and $P_{r(f)}T_f = T_f$ for all $f \in \Lambda^1$. 
\end{lemma}

Before proving the lemma we note that, combined with the orthogonality of the projections, the lemma implies that multiplication in $C^*(\Lambda)$ is compatible with concatenation in $\Lambda$. 

\begin{proof}
Let $\{P_v,T_f\}$ be a Cuntz-Krieger $\Lambda$-family. Then $T_fP_{s(f)} = T_fT_f^*T_f = T_f,$ where the first equality follows from \ref{CK3} and the second is because $T_f$ is a partial isometry.

Now we will show $P_{r(f)}T_f = T_f$. Let $d(f) = \mathbf{e}_i$. By \ref{CK4} and the definition of partial isometry, we know \[P_{r(f)}T_f = \left(\sum\limits_{e\in r(f)\Lambda^{\mathbf{e}_i}}T_eT_e^*\right)T_f = \left(\sum\limits_{e\in r(f)\Lambda^{\mathbf{e}_i}}T_eT_e^*\right)(T_fT_f^*T_f).\] 
Next, observe that $T_fT_f^*$ will appear as a term in the sum. Since $P_{r(f)}$ is a projection, and each term in the sum is a projection, the terms in the sum must be orthogonal projections. Therefore, $T_eT_e^*T_fT_f^* = 0$ unless $e= f$. By substitution we obtain $\left(\sum\limits_{e\in r(f)\Lambda^{\mathbf{e}_i}}T_eT_e^*\right)(T_fT_f^*T_f) = T_fT_f^*T_fT_f^*T_f$, which is equal to $T_f$ because $T_f$ is a partial isometry.
\end{proof}

\subsection{Kumjian-Pask algebras}
Kumjian-Pask algebras are an algebraic analogue of higher rank graph $C^*$-algebras. We provide a brief overview of Kumjian-Pask algebras here. For a more detailed introduction, we refer the reader to \cite{KPAlgebras}. 

\begin{defn} \label{ghostpath} \cite[Section 3]{KPAlgebras}
Let $\Lambda$ be a $k$-graph and $\Lambda^{\neq \mathbf{0}} := \{\lambda \in \Lambda : d(\lambda) \neq \mathbf{0}\}$. For each $\lambda \in \Lambda^{\neq \mathbf{0}}$, there exists a \textit{ghost path}, denoted $\lambda^*$. If $v \in \Lambda^0$, $v^* = v$. For $\lambda \in \Lambda^{\neq \mathbf{0}}$, we set $d(\lambda^*) = -d(\lambda)$, $r(\lambda^*) = s(\lambda)$, and $s(\lambda^*) = r(\lambda)$. For the set of ghost paths, we write $G(\Lambda)$. For composition on $G(\Lambda)$, we set $\lambda^*\mu^* = (\mu\lambda)^*$ where $\lambda, \mu \in \Lambda^{\neq \mathbf{0}}$ and $r(\mu^*) = s(\lambda^*)$. The factorization rules of $\Lambda$ induce similar factorization rules on $G(\Lambda)$.
\end{defn}

\begin{defn}\label{KPfamily} \cite[Definition 3.1]{KPAlgebras}
Let $\Lambda$ be a row-finite, source-free $k$-graph, and let $R$ be a commutative ring with $1$. A \textit{Kumjian-Pask $\Lambda$-family} $(P,T)$ in an $R$-algebra $A$ consists of two functions $P: \Lambda^0 \to A$ and $T: \Lambda^{\neq \mathbf{0}} \cup G(\Lambda^{\neq \mathbf{0}}) \to A$ such that:

\begin{enumerate}
\myitem{(KP1)}\label{KP1} $\{P_v: v \in \Lambda^0\}$ is a family of mutually orthogonal idempotents,

\myitem{(KP2)}\label{KP2} for all $\lambda, \mu \in \Lambda^{\neq \mathbf{0}}$ with $r(\mu) = s(\lambda)$, we have
\begin{align*}
& T_{\lambda}T_{\mu}= T_{\lambda\mu}, \quad  T_{\mu^*}T_{\lambda^*}= T_{(\lambda\mu)^*}, \quad 
 P_{r(\lambda)}T_{\lambda} = T_{\lambda} = T_{\lambda}P_{s(\lambda)}, \\
 & P_{s(\lambda)}T_{\lambda^*} = T_{\lambda^*} = T_{\lambda^*}P_{r(\lambda)},
\end{align*}
\myitem{(KP3)}\label{KP3} for all $\lambda, \mu \in \Lambda^{\neq \mathbf{0}}$ with $d(\lambda) = d(\mu)$, we have $T_{\lambda^*}T_{\mu} = \delta_{\lambda, \mu}P_{s(\lambda)}$
\myitem{(KP4)}\label{KP4} for all $v \in \Lambda^0$ and all $\mathbf{n} \in \NN^k \setminus \{\mathbf{0}\}$, we have $P_v = \sum_{\lambda \in v\Lambda^{\mathbf{n}}}T_{\lambda}T_{\lambda^*}.$
\end{enumerate}

While we have distinguished the vertex idempotents to emphasize there is only one generator for each path of degree $\mathbf{0}$, while paths of nonzero degree have two, it is convenient when writing formulas to allow $T_v := P_v$ and $T_{v^*} : = P_v$.

\end{defn}

The following lemma asserts that it suffices to define a Kumjian-Pask family only on vertices and edges. 

\begin{lem}\label{lem-KPprime}
Let $\Lambda$ be a source-free, row-finite $k$-graph and let $R$ be a commutative ring with $1$.  Suppose $\{ P_{v} , T_e, T_{e^*} : v \in \Lambda^0 , e \in \Lambda^1 \}$ is a collection of elements in an $R$-algebra $A$ such that 
\begin{enumerate}
    \myitem{(KP1')}\label{KP1'} for all $v, w \in \Lambda^0$, $P_v P_w = \delta_{v,w} P_v$;
    \myitem{(KP2')}\label{KP2'} For all $a,x,b,y \in \Lambda^1$ such that $ax \sim by$, 
    \begin{align*}
        P_{r(a)} T_a &= T_a P_{s(a)} =T_a & P_{s(a)} T_{a^*} &= T_{a^*} P_{r(a)} = T_{a^*} 
    \end{align*}
    and
    \begin{align*}
        T_{a}T_x &= T_b T_y & T_{x^*} T_{a^*} &= T_{y^*} T_{b^*};
    \end{align*}

    \myitem{(KP3')}\label{KP3'} For all $e, f \in \Lambda^1$, $T_{e^*} T_f = \delta_{e,f} P_{s(e)}$; and 

    \myitem{(KP4')}\label{KP4'} For all $v \in \Lambda^0$ and for all $1 \leq i \leq k$, 
    \[P_v = \sum_{ e \in  v\Lambda^{\mathbf{e}_i} } T_{e} T_{e^*}.\]
\end{enumerate}
Set $T_\mu = T_{\mu_n} \cdots T_{\mu_1}$ and $T_{\mu^*} = T_{\mu_1^*} \cdots T_{\mu_n^*}$ for all $\mu = \mu_n \cdots \mu_1 \in \Lambda$ where $\mu_i \in \Lambda^1$ for each $i$.  Then $\{ P_v , T_\mu : v \in \Lambda^0 , \mu \in \Lambda\}$ is a Kumjian-Pask $\Lambda$-family in $A$.
\end{lem}

\begin{proof}
    The relations \ref{KP1} - \ref{KP3} can easily be verified. We prove \ref{KP4} by induction. Fix $v \in \Lambda^0$ and $\mathbf{n} \in \mathbb{N}^k \setminus \{0\}$. If $|\mathbf{n}|=1$, then \ref{KP4} follows immediately. Assume \ref{KP4} holds for $|\mathbf{n}| <N$, and consider  $|\mathbf{n}| = N$. Then there exists $i$ such that $\mathbf{e}_i \leq \mathbf{n}$. For every $\lambda \in v\Lambda^{\mathbf{n}}$, we can write $\lambda = \lambda_1 \tilde{\lambda}$, where $d(\lambda_1) = \mathbf{e}_i$ and $d(\tilde{\lambda}) = \mathbf{n}-\mathbf{e}_i$. Then
    \begin{align*}
        \sum_{\lambda \in v\Lambda^{\mathbf{n}}} T_\lambda T_{\lambda^*} &= \sum_{\lambda_1 \in v\Lambda^{\mathbf{e}_i}} \sum_{\tilde{\lambda} \in s(\lambda_1)\Lambda^{\mathbf{n} - \mathbf{e}_i} } T_{\lambda_1}T_{ \tilde{\lambda}} T_{\tilde{\lambda}^*}T_{\lambda_1^*} \\
        &=\sum_{\lambda_1 \in v\Lambda^{\mathbf{e}_i}} T_{\lambda_1} \left(\sum_{\tilde{\lambda} \in s(\lambda_1)\Lambda^{\mathbf{n} - \mathbf{e}_i} } T_{ \tilde{\lambda}} T_{\tilde{\lambda}^*}\right) T_{\lambda_1^*} \\
        &= \sum_{\lambda_1 \in v\Lambda^{\mathbf{e}_i}} T_{\lambda_1} P_{s(\lambda_1)}T_{\lambda_1^*} \quad \text{by induction hypothesis} \\
        &= \sum_{\lambda_1 \in v\Lambda^{\mathbf{e}_i}} T_{\lambda_1} T_{\lambda_1^*} \quad \text{by \ref{KP2'}} \\
        &= P_v \quad \text{by \ref{KP4'}} \qedhere
    \end{align*} \end{proof}

\begin{thm}\label{KPHom}\cite[Theorem 3.4]{KPAlgebras}
Let $\Lambda$ be a row-finite, source-free $k$-graph, and let $R$ be a commutative ring with $1$. Then there is an $R$-algebra $\mathrm{KP}_R(\Lambda)$ generated by a Kumjian-Pask $\Lambda$-family $(p,t)$, such that, whenever $(P,T)$ is a Kumjian-Pask $\Lambda$-family in an $R$-algebra $A$, there is a unique $R$-algebra homomorphism $\pi_{P,T}: \mathrm{KP}_R(\Lambda) \to A$ such that
\begin{equation*}
\pi_{P,T}(p_v) = P_v, \quad \pi_{P,T}(t_{\lambda}) = T_{\lambda}, \quad \pi_{P,T}(t_{\mu^*}) = T_{\mu^*}
\end{equation*}
for $v \in \Lambda^0$ and $\lambda, \mu \in \Lambda^{\neq \mathbf{0}}$. There is a $\mathbb{Z}^k$-grading on $\mathrm{KP}_R(\Lambda)$ satisfying
\begin{equation*}
\mathrm{KP}_R(\Lambda)_{\mathbf{n}} = span_R\{t_{\lambda}t_{\mu^*} : \lambda, \mu \in \Lambda \text{ and } d(\lambda) - d(\mu) = \mathbf{n}\},
\end{equation*}
and we have $rp_v \neq 0$ for $v \in \Lambda^0$ and $r \in R\setminus \{0\}$.
\end{thm}

We call $\mathrm{KP}_R(\Lambda)$ the \textit{Kumjian-Pask algebra} of $\Lambda$ and $(p,t)$ the \textit{universal Kumjian-Pask $\Lambda$-family}. Whenever possible, we will use lowercase letters to denote the universal Kumjian-Pask $\Lambda$-family.

\section{LiMaR-split for higher rank graphs} \label{sec: limar-split}

In this section, we describe how to LiMaR-split a $k$-graph $\Lambda$ in a degree $B$, and we prove that the result is also a $k$-graph. We begin by introducing a set $S_w$ of vertices that will get split and a function $n$ which determines how many copies there will be of edges and vertices in the LiMaR-split graph.
\begin{defn}[$n$ function] \label{def: nfunction}
Let $(\Lambda, d)$ be a $k$-graph and $G=(\Lambda^{0},\Lambda^1,r,s)$ its $1$-skeleton. Let the color set $C$ be the set of canonical basis elements of $\NN^{k}$. Let $B\in C$, which we call the \textit{degree of the LiMaR-Split}, and let $A=C\setminus \{B\}$. Fix $w\in\Lambda^{0}$ such that $|s_B^{-1}(w)|\geq 2$; that is, at least two edges are coming out of the vertex $w$ of color $B$. Let $S_w$ be the minimal subset of $\Lambda^0$ satisfying the following conditions:
\begin{enumerate}
    \item \label{S_w1} $w \in S_w$
    \item \label{S_w2} If $x \in S_w$ and there exists $e \in \Lambda^A$ such that $s(e)=x$ and $|s_B^{-1}(r(e))|\geq 2$, then $r(e) \in S_w$.
\end{enumerate}
We define the function $n:\Lambda^0\rightarrow \NN$ by
\begin{gather*}
 n(v)=
\begin{cases}
    |s_B^{-1}(v)|, & v \in S_w \\
    1,  & v \not\in S_w 
\end{cases}
\end{gather*}
\end{defn}

\begin{rmk}
    By definition of $S_w$, we have $n(v) \geq 2$ if and only if $v \in S_w$. Moreover, if $v \in S_w$, then $n(v)$ equals the number of edges of color $B$ coming out of $v$.
\end{rmk}

\begin{defn}[$1$-skeleton of LiMaR-split]\label{def: 1-skeletonlimarsplit}
Let $(\Lambda, d)$ be a $k$-graph and $G=(\Lambda^{0},\Lambda^1,r,s)$ its $1$-skeleton. Fix $B \in C$ and $w \in \Lambda^0$ as in Definition \ref{def: nfunction}. We define the \textit{$1$-skeleton of the LiMaR-split of $\Lambda$ (in degree $B$)}, denoted $G_L=\{\Gamma^{0},\Gamma^1,r_{\Gamma},s_{\Gamma}\}$, as follows. 

For every $v\in\Lambda^{0}$ with $s_B^{-1}(v) \neq \emptyset$, partition $s^{-1}_B(v)$ into $n(v)$ nonempty sets $\mathcal{E}^{v}_1 \sqcup \mathcal{E}^{v}_2 \sqcup \ldots \sqcup\mathcal{E}^{v}_{n(v)}$. 

Define
\begin{align*}
\Gamma^{0}&=\{v^i : v \in \Lambda^{0},1\leq i\leq n(v)\}\\
\Gamma^{1}&=\{e^{i} : e \in \Lambda^{1}, 1 \leq i \leq n(r(e))\}.
\end{align*}

For $e^i \in \Gamma^1$, we define the range and source maps as follows. The source map is defined piecewise in three cases. The first case is used when $d(e)= B$. The second case is used when $d(e)\in A$ and $|s^{-1}_{B}(r(e))|\geq 1$. The last condition accounts for the case where $d(e) \in A$ and $|s^{-1}_{B}(r(e))|=0$:
\begin{align}
\label{rangemap} r_\Gamma(e^{i}) &= r(e)^i\\
\label{sourcemap} s_{\Gamma}(e^i) &= \begin{cases}
 s(e)^j & \text{if } e\in\mathcal{E}^{s(e)}_j;\\
 s(e)^j & \text{if } fe \in\Lambda \text{, } d(f)\neq d(e) \text{, } f\in\mathcal{E}_i^{s(f)}\text{, }fe\sim ac, \text{ and }c\in\mathcal{E}_j^{s(c)};\\
s(e)^1  & \text{else}.
\end{cases}
\end{align}
\end{defn}

\begin{lemma}
    The source and range maps $s_\Gamma$ and $r_\Gamma$ in Definition \ref{def: 1-skeletonlimarsplit} are well-defined.
\end{lemma}

\begin{proof}
    Let $e \in \Lambda^1$ with $r(e) = v$, and fix $1 \leq i \leq n(v)$. Since there are $n(v)$ copies of $v$ in $\Gamma^0$ and $n(v)$ copies of $e$ in $\Gamma^1$, then $r_\Gamma(e^i)$ is well-defined. 
    
    Now we check that $s_\Gamma(e^i)$ is well-defined. If $d(e) = B$, then $e \in \mathcal{E}_j^{s(e)}$ for some $j$ such that $1 \leq j \leq n(s(e))$. So $s_\Gamma(e^i)$ is well-defined. If $d(e) \neq B$, then we have two cases: $v$ is not a degree $B$ sink or $v$ is a degree $B$ sink. If $v$ is not a degree $B$ sink, then $s_B^{-1}(v)$ is partitioned into $n(v)$ nonempty sets $\mathcal{E}^{v}_1 \sqcup \mathcal{E}^{v}_2 \sqcup \ldots \sqcup\mathcal{E}^{v}_{n(v)}$. If $|\mathcal{E}_i^v|=1$, let $\{f\} = \mathcal{E}_i^v$. Then there exists a unique path $ac \in \Lambda$ such that $fe \sim ac$ and $c \in s_B^{-1}(s(e))$. Since $|s_B^{-1}(s(e))| \neq \emptyset$, then $s_B^{-1}(s(e))$ is partitioned into $n(s(e))$ nonempty sets, exactly one of which contains $c$. Therefore $s_\Gamma(e^i)$ is well-defined. On the other hand, if $|\mathcal{E}_i^v| \geq 2$, then there exist $f, f' \in \mathcal{E}_i^v$. This can only happen if $n(v)=1$, which means $v \not\in S_w$. Since $\Lambda$ is a $k$-graph, there exist unique paths $ac$ and $a'c'$ such that $fe \sim ac$ and $f'e\sim a'c'$ and $c, c' \in s_B^{-1}(s(e))$. Since $|s_B^{-1}(s(e))| \neq \emptyset$, then $s_B^{-1}(s(e))$ is partitioned into $n(s(e))$ nonempty sets.  However, since $v=r(e) \not\in S_w$, then $s(e) \not\in S_w$ by \ref{S_w2}, so $n(s(e))=1$, which implies both $c, c' \in \mathcal{E}_1^{s(e)}$, so $s_\Gamma(e^i)$ is well-defined. Lastly, if $v$ is a degree $B$ sink and $d(e) \neq B$, then only the third condition in Equation \eqref{sourcemap} applies, and $s_\Gamma(e^i)$ is well-defined. 
\end{proof}

\begin{defn}[Parent Function]\label{def: parent function}\cite[Definition 3.5.1]{efgggp} 
    Let $G$ be the $1$-skeleton of a $k$-graph $(\Lambda,d)$, and let $G_L$ be the $1$-skeleton of the LiMaR-split of $\Lambda$, as defined in Definition \ref{def: 1-skeletonlimarsplit}. We define a morphism    $par:G_L\rightarrow G$ by $par(e^i)=e$ and $par(v^i)=v$ for any $i\in\NN$  and $e^i\in\Gamma^1$ or $v^i\in\Gamma^0$. For any $e,f\in \Gamma$ with $par(e)=par(f)$, we say that $e$ and $f$ are \textit{copies of $par(e)$} in $\Gamma$.
\end{defn}

\begin{defn}[Degree Map $d_\Gamma$]\label{def: degreemaplimar}
Let $G$ be the $1$-skeleton of a $k$-graph $(\Lambda,d)$, and let $G_L$ be the $1$-skeleton of the LiMaR-split of $\Lambda$. For every $e \in \Gamma^1$, we define the degree map $d_\Gamma: \Gamma^1 \to \mathbb{N}^k$ by $d_{\Gamma}(e)$ to be $d(par(e))$.
\end{defn}

\begin{defn}[Equivalence Relation $\sim_L$]\label{def: equivrelationlimar}
Let $(\Lambda,d)$ be a $k$-graph and $G_L$ be the $1$-skeleton of the LiMaR-split of $\Lambda$. We define an equivalence relation $\sim_{L}$ on $G_L$ by $\lambda \sim_{L} \mu$ if $par(\lambda)\sim par(\mu)$ and $r_{\Gamma}(\lambda)=r_{\Gamma}(\mu)$, $s_{\Gamma}(\lambda)=s_{\Gamma}(\mu)$, for $\lambda,\mu\in G_L$.
\end{defn}

\begin{defn}[LiMaR-split of $\Lambda$]
    Let $(\Lambda,d)$ be a $k$-graph and $G_L$ be the $1$-skeleton of the LiMaR-split of $\Lambda$, as defined in Definition \ref{def: 1-skeletonlimarsplit}. We define $\Gamma=G_L/\sim_L$ together with the degree map $d_\Gamma$ to be the \textit{LiMaR-split of $\Lambda$ (in degree $B$)}.
\end{defn}

Before proving that $(\Gamma, d_\Gamma)$ is indeed a $k$-graph, we provide  a couple of examples of LiMaR-split applied to a $2$-graph.

\begin{example} \label{ex:princessleiaplus}
Let $\Lambda$ be the $2$-graph with $1$-skeleton and factorization rules below:

\begin{tikzpicture}[scale = 0.7]
\node[inner sep=1pt, outer sep = 3pt, circle, fill = black] (v) [label=above:{\small$v$}] at (-3,0) {};
\node[inner sep=1pt, outer sep = 3pt,circle, fill = black] (w) [label=below:{\small$x$}] at (0,0) {};
\node[inner sep=1pt, outer sep = 3pt,circle, fill = black] (x) [label=above:{\small$z$}] at (3,0) {};
\node[inner sep=1pt, outer sep = 3pt,circle, fill = black] (y) [label=left:{\small$y$}] at (0,2) {};

\draw [red,dashed, -latex] (v) edge[out = -50, in = -130] node[midway, below] {$c$} (x);
\draw [red,dashed, -latex] (v) edge[loop left, out = 120, in = -120, distance = 4cm] node {$\beta$} (v);
\draw[red,dashed, -latex] (v)  edge[out = -30, in = -150] node [midway, above] {$b$} (w);
\draw[red,dashed, -latex] (w) edge [out = -30,in = -150] node [midway, above] {$f$} (x);
\draw[red, dashed, -latex] (x) edge [loop right, out = -120, in = 120, distance = -4cm] node {$m$} (x);
\draw[red,dashed, -latex] (w) edge [out = 50,in = -50] node [midway, right] {$e$} (y);

\draw [blue, -latex] (v) edge[loop left, out = 120, in = -120,distance = 2cm] node {$\alpha$} (v);
\draw[blue,-latex] (v) edge [out = -90, in = -90] node [midway, below] {$i$} (x);
\draw[blue, -latex] (v) edge [out = 25, in = 150] node [midway, above] {$h$} (w);
\draw[blue, -latex] (w) edge [out = 25, in = 160] node [midway, above] {$\ell$} (x);
\draw[blue, -latex] (x) edge[loop right, out = 60, in = -60,distance = 2cm] node{$n$} (x);
\draw[blue, -latex] (w) edge [out = -225, in = -130] node [midway, left] {$k$} (y);

\node at (-3.5,3) {$\Lambda:$};
\node at (9,3) {\small 1. ${\color{red}\beta}{\color{blue}\alpha}={\color{blue}\alpha}{\color{red}\beta}$};
\node at (9,2.5) {\small 2. ${\color{red}b}{\color{blue}\alpha}={\color{blue}h}{\color{red}\beta}$};
\node at (9,2) {\small 3. ${\color{red}c}{\color{blue}\alpha}={\color{blue}i}{\color{red}\beta}$};
\node at (9,1.5) {\small 4. ${\color{red}e}{\color{blue}h}={\color{blue}k}{\color{red}b}$};
\node at (9,1) {\small 5. ${\color{red}f}{\color{blue}h}={\color{blue}\ell}{\color{red}{b}}$};
\node at (9,0.5) {\small 6. ${\color{red}m}{\color{blue}\ell}={\color{blue}n}{\color{red}f}$};
\node at (9,0) {\small 7. ${\color{red}m}{\color{blue}n}={\color{blue}n}{\color{red}m}$};
\node at (9,-0.5) {\small 8. ${\color{red}m}{\color{blue}i}={\color{blue}n}{\color{red}c}$};
\end{tikzpicture}

Let $B$ refer to the blue edges, and fix the vertex $v$. Then $S_v = \{v,x\}$ so $n(v)=3,n(x)=2,n(y)=1,n(z)=1$. The $1$-skeleton of the LiMaR-split of $\Lambda$ derived from the partitions $\mathcal{E}^v_1=\{\alpha\}$, $\mathcal{E}^v_2=\{h\}$, $\mathcal{E}^v_3=\{i\}$, $\mathcal{E}^x_1=\{k\}$, $\mathcal{E}^x_2=\{\ell\}$, and $\mathcal{E}^z_1=\{n\}$ is a $2$-graph when equipped with a degree map as defined in Definition \ref{def: degreemaplimar} and the following factorization rules:

\begin{tikzpicture}[scale = 0.7]
\node[inner sep=1pt, outer sep = 3pt, circle, fill = black] (v1) [label=above:{\small$v^1$}] at (-1,1) {};
\node[inner sep=1pt, outer sep = 3pt, circle, fill = black] (v2) [label=left:{\small$v^2$}] at (-1,-1) {};
\node[inner sep=1pt, outer sep = 3pt, circle, fill = black] (v3) [label=above:{\small$v^3$}] at (-1,-3) {};
\node[inner sep=1pt, outer sep = 3pt,circle, fill = black] (x1) [label=above:{\small$x^1$}] at (3,0.5) {};
\node[inner sep=1pt, outer sep = 3pt,circle, fill = black] (x2) [label=above:{\small$x^2$}] at (3,-4) {};
\node[inner sep=1pt, outer sep = 3pt,circle, fill = black] (z1) [label=right:{\small$z^1$}] at (5,-5) {};
\node[inner sep=1pt, outer sep = 3pt,circle, fill = black] (y1) [label=left:{\small$y^1$}] at (5,0.5) {};

\node at (-3,3) {$\Gamma:$};

\node at (9,3) {\small 1. ${\color{red}\beta^1}{\color{blue}\alpha^1}={\color{blue}\alpha^1}{\color{red}\beta^1}$};
\node at (9,2.5) {\small 2. ${\color{red}\beta^2}{\color{blue}\alpha^1}={\color{blue}\alpha^2}{\color{red}\beta^1}$};
\node at (9,2) {\small 3. ${\color{red}\beta^3}{\color{blue}\alpha^1}={\color{blue}\alpha^3}{\color{red}\beta^1}$};
\node at (9,1.5) {\small 4. ${\color{red}b^1}{\color{blue}\alpha^2}={\color{blue}h^1}{\color{red}\beta^2}$};
\node at (9,1) {\small 5. ${\color{red}b^2}{\color{blue}\alpha^2}={\color{blue}h^2}{\color{red}\beta^2}$};
\node at (9,0.5) {\small 6. ${\color{red}c^1}{\color{blue}\alpha^3}={\color{blue}i^1}{\color{red}\beta^3}$};
\node at (9,0) {\small 7. ${\color{red}e^1}{\color{blue}h^1}={\color{blue}k^1}{\color{red}b^1}$};
\node at (9,-0.5) {\small 8. ${\color{red}f^1}{\color{blue}h^2}={\color{blue}\ell^1}{\color{red}b^2}$};
\node at (9,-1) {\small 9. ${\color{red}m^1}{\color{blue}\ell^1}={\color{blue}n^1}{\color{red}f^1}$};
\node at (9,-1.5) {\small 10. ${\color{red}m^1}{\color{blue}n^1}={\color{blue}n^1}{\color{red}m^1}$};
\node at (9,-2) {\small 11. ${\color{red}m^1}{\color{blue}i^1}={\color{blue}n^1}{\color{red}c^1}$};

\draw [red,dashed, -latex] (v3) edge[out = -60, in = 210, distance = 2cm] node[midway, below] {$c^1$} (z1);
\draw [red,dashed, -latex] (v1) edge[loop left, out = 65, in = -210, distance = 2cm] node {$\beta^1$} (v1);
\draw [red,dashed, -latex] (v1) edge[loop left, out = 230, in = 130] node {$\beta^2$} (v2);
\draw [red,dashed, -latex] (v1) edge[loop left, out = -180, in = 160, distance = 3cm] node {$\beta^3$} (v3);
\draw[red,dashed, -latex] (v2)  edge[out = 0, in = -150] node [midway, below] {$b^1$} (x1);
\draw[red,dashed, -latex] (v2)  edge[out = -30, in = -230] node [midway, below] {$b^2$} (x2);
\draw[red,dashed, -latex] (x2) edge [out = -60,in = -190] node [midway, below] {$f^1$} (z1);
\draw[red, dashed, -latex] (z1) edge [loop right, out = 80, in = 145, distance = -2cm] node {$m^1$} (z1);
\draw[red,dashed, -latex] (x1) edge [out = 50,in = -220] node [midway, above] {$e^1$} (y1);

\draw [blue, -latex] (v1) edge[loop left, out = 40, in = -40,distance = 2cm] node {$\alpha^1$} (v1);
\draw [blue, -latex] (v1) edge[loop left, out = -60, in = 60] node {$\alpha^2$} (v2);
\draw [blue, -latex] (v1) edge[loop left, out = -160, in = 140, distance = 2cm] node {$\alpha^3$} (v3);
\draw[blue,-latex] (v3) edge [out = -90, in = -120] node [midway, below] {$i^1$} (z1);
\draw[blue, -latex] (v2) edge [out = 40, in = 170] node [midway, below] {$h^1$} (x1);
\draw[blue, -latex] (v2) edge [out = -50, in = -200] node [midway, below] {$h^2$} (x2);
\draw[blue, -latex] (x2) edge [out = -10, in = 130] node [midway, above] {$\ell^1$} (z1);
\draw[blue, -latex] (z1) edge[loop right, out = 90, in = 20,distance = 2cm] node{$n^1$} (z1);
\draw[blue, -latex] (x1) edge [out = -50, in = -130] node [midway, below] {$k^1$} (y1);
\end{tikzpicture}

\end{example}

A difference in factorization rules  for $\Lambda$ can result in a different $2$-graph  $\Gamma$.

\begin{example} \label{ex:princessleiaplus2}
    Consider $\Lambda$ with the same $1$-skeleton as depicted above but with the alternate factorization rules:\\

\begin{tikzpicture}[scale = 0.7]
\node[inner sep=1pt, outer sep = 3pt, circle, fill = black] (v) [label=above:{\small$v$}] at (-3,0) {};
\node[inner sep=1pt, outer sep = 3pt,circle, fill = black] (w) [label=below:{\small$x$}] at (0,0) {};
\node[inner sep=1pt, outer sep = 3pt,circle, fill = black] (x) [label=above:{\small$z$}] at (3,0) {};
\node[inner sep=1pt, outer sep = 3pt,circle, fill = black] (y) [label=left:{\small$y$}] at (0,2) {};

\node at (-3.5,3) {$\Lambda:$};

\draw [red,dashed, -latex] (v) edge[out = -50, in = -130] node[midway, below] {$c$} (x);
\draw [red,dashed, -latex] (v) edge[loop left, out = 120, in = -120, distance = 4cm] node {$\beta$} (v);
\draw[red,dashed, -latex] (v)  edge[out = -30, in = -150] node [midway, above] {$b$} (w);
\draw[red,dashed, -latex] (w) edge [out = -30,in = -150] node [midway, above] {$f$} (x);
\draw[red, dashed, -latex] (x) edge [loop right, out = -120, in = 120, distance = -4cm] node {$m$} (x);
\draw[red,dashed, -latex] (w) edge [out = 50,in = -50] node [midway, right] {$e$} (y);

\draw [blue, -latex] (v) edge[loop left, out = 120, in = -120,distance = 2cm] node {$\alpha$} (v);
\draw[blue,-latex] (v) edge [out = -90, in = -90] node [midway, below] {$i$} (x);
\draw[blue, -latex] (v) edge [out = 25, in = 150] node [midway, above] {$h$} (w);
\draw[blue, -latex] (w) edge [out = 25, in = 160] node [midway, above] {$\ell$} (x);
\draw[blue, -latex] (x) edge[loop right, out = 60, in = -60,distance = 2cm] node{$n$} (x);
\draw[blue, -latex] (w) edge [out = -225, in = -130] node [midway, left] {$k$} (y);

\node at (9,3) {\small 1. ${\color{red}\beta}{\color{blue}\alpha}={\color{blue}\alpha}{\color{red}\beta}$};
\node at (9,2.5) {\small 2. ${\color{red}b}{\color{blue}\alpha}={\color{blue}h}{\color{red}\beta}$};
\node at (9,2) {\small 3. ${\color{red}c}{\color{blue}\alpha}={\color{blue}i}{\color{red}\beta}$};
\node at (9,1.5) {\small 4. ${\color{red}e}{\color{blue}h}={\color{blue}k}{\color{red}b}$};
\node at (9,1) {\small 5. ${\color{red}f}{\color{blue}h}={\color{blue}n}{\color{red}{c}}$};
\node at (9,0.5) {\small 6. ${\color{red}m}{\color{blue}\ell}={\color{blue}n}{\color{red}f}$};
\node at (9,0) {\small 7. ${\color{red}m}{\color{blue}n}={\color{blue}n}{\color{red}m}$};
\node at (9,-0.5) {\small 8. ${\color{red}m}{\color{blue}i}={\color{blue}\ell}{\color{red}b}$};
\end{tikzpicture}

These factorization rules together with  the same partitions as above ($\mathcal{E}^v_1=\{\alpha\}$, $\mathcal{E}^v_2=\{h\}$, $\mathcal{E}^v_3=\{i\}$, $\mathcal{E}^x_1=\{k\}$, $\mathcal{E}^x_2=\{\ell\}$, and $\mathcal{E}_1^z = \{n\}$) give rise to the following $2$-graph:

\begin{center}
        \begin{tikzpicture}[scale = 0.7]
\node[inner sep=1pt, outer sep = 3pt, circle, fill = black] (v1) [label=above:{\small$v^1$}] at (-1,1) {};
\node[inner sep=1pt, outer sep = 3pt, circle, fill = black] (v2) [label=left:{\small$v^2$}] at (-1,-1) {};
\node[inner sep=1pt, outer sep = 3pt, circle, fill = black] (v3) [label=above:{\small$v^3$}] at (-1,-3) {};
\node[inner sep=1pt, outer sep = 3pt,circle, fill = black] (x1) [label=above:{\small$x^1$}] at (3,0.5) {};
\node[inner sep=1pt, outer sep = 3pt,circle, fill = black] (x2) [label=above:{\small$x^2$}] at (3,-4) {};
\node[inner sep=1pt, outer sep = 3pt,circle, fill = black] (z1) [label=right:{\small$z^1$}] at (5,-5) {};
\node[inner sep=1pt, outer sep = 3pt,circle, fill = black] (y1) [label=left:{\small$y^1$}] at (5,0.5) {};

\node at (-3,3) {$\Gamma:$};

\node at (9,3) {\small 1. ${\color{red}\beta^1}{\color{blue}\alpha^1}={\color{blue}\alpha^1}{\color{red}\beta^1}$};
\node at (9,2.5) {\small 2. ${\color{red}\beta^2}{\color{blue}\alpha^1}={\color{blue}\alpha^2}{\color{red}\beta^1}$};
\node at (9,2) {\small 3. ${\color{red}\beta^3}{\color{blue}\alpha^1}={\color{blue}\alpha^3}{\color{red}\beta^1}$};
\node at (9,1.5) {\small 4.
${\color{red}b^1}{\color{blue}\alpha^2}={\color{blue}h^1}{\color{red}\beta^2}$};
\node at (9,1) {\small 5.
${\color{red}b^2}{\color{blue}\alpha^3}={\color{blue}h^2}{\color{red}\beta^2}$};
\node at (9,0.5) {\small 6. ${\color{red}c^1}{\color{blue}\alpha^2}={\color{blue}i^1}{\color{red}\beta^3}$};
\node at (9,-0) {\small 7. ${\color{red}e^1}{\color{blue}h^1}={\color{blue}k^1}{\color{red}b^1}$};
\node at (9,-0.5) {\small 8. ${\color{red}f^1}{\color{blue}h^2}={\color{blue}n^1}{\color{red}c^1}$};
\node at (9,-1) {\small 9. ${\color{red}m^1}{\color{blue}\ell^1}={\color{blue}n^1}{\color{red}f^1}$};
\node at (9,-1.5) {\small 10. ${\color{red}m^1}{\color{blue}n^1}={\color{blue}n^1}{\color{red}m^1}$};
\node at (9,-2) {\small 11. ${\color{red}m^1}{\color{blue}i^1}={\color{blue}\ell^1}{\color{red}b^2}$};

\draw [red,dashed, -latex] (v2) edge[out = -20, in = 105, distance = 2cm] node[midway, below] {$c^1$} (z1);
\draw [red,dashed, -latex] (v1) edge[loop left, out = 65, in = -210, distance = 2cm] node {$\beta^1$} (v1);
\draw [red,dashed, -latex] (v1) edge[loop left, out = 240, in = 130] node {$\beta^2$} (v2);
\draw [red,dashed, -latex] (v1) edge[loop left, out = -180, in = 160, distance = 3cm] node {$\beta^3$} (v3);
\draw[red,dashed, -latex] (v2)  edge[out = 0, in = -150] node [midway, below] {$b^1$} (x1);
\draw[red,dashed, -latex] (v3)  edge[out = -50, in = -150] node [midway, below] {$b^2$} (x2);
\draw[red,dashed, -latex] (x2) edge [out = -60,in = -190] node [midway, below] {$f^1$} (z1);
\draw[red, dashed, -latex] (z1) edge [loop right, out = 80, in = 150, distance = -2cm] node {$m^1$} (z1);
\draw[red,dashed, -latex] (x1) edge [out = 50,in = -220] node [midway, above] {$e^1$} (y1);

\draw [blue, -latex] (v1) edge[loop left, out = 40, in = -40,distance = 2cm] node {$\alpha^1$} (v1);
\draw [blue, -latex] (v1) edge[loop left, out = -60, in = 60] node {$\alpha^2$} (v2);
\draw [blue, -latex] (v1) edge[loop left, out = -160, in = 140, distance = 2cm] node {$\alpha^3$} (v3);
\draw[blue,-latex] (v3) edge [out = -90, in = -120] node [midway, below] {$i^1$} (z1);
\draw[blue, -latex] (v2) edge [out = 40, in = 170] node [midway, below] {$h^1$} (x1);
\draw[blue, -latex] (v2) edge [out = -50, in = -200] node [midway, below] {$h^2$} (x2);
\draw[blue, -latex] (x2) edge [out = -10, in = 130] node [midway, above] {$\ell^1$} (z1);
\draw[blue, -latex] (z1) edge[loop right, out = 90, in = 20,distance = 2cm] node{$n^1$} (z1);
\draw[blue, -latex] (x1) edge [out = -50, in = -130] node [midway, below] {$k^1$} (y1);
\end{tikzpicture}
    \end{center}
\end{example}

We will use the following remark and lemmas to prove that, under one additional assumption, the $k$-colored graph $\Gamma =G_L^*/\sim_{L}$ that results from LiMaR-splitting is indeed a $k$-graph. 

\begin{rmk} \label{rmk: distinctrange}
The definition of LiMaR-split guarantees that distinct copies in $\Gamma$ of an edge in $\Lambda$ have distinct ranges.
\end{rmk}

\begin{lemma} \label{lemma: parentpaths}
Suppose $(\Gamma,d_\Gamma)$ is the result of LiMaR-splitting a $k$-graph $(\Lambda,d)$. If $e^if^j$ is a two edge path in $G_L^*$, then $par(e^i)par(f^j)$ is a two edge path in $G^*$ of the same degree.
\end{lemma}

\begin{proof}
Let $e^if^j$ be a two edge path in $G_L^*$, which implies that $r_{\Gamma}(f^j) = s_{\Gamma}(e^i)$. Let $e:=par(e^i)$ and $f:= par(f^j)$. By Definition \ref{def: 1-skeletonlimarsplit}, $r_{\Gamma}(f^j) = r(f)^j$ and $s_{\Gamma}(e^i) = s(e)^p$ for some $p \in \NN^+$. Then, by substitution, $r(f)^j = s(e)^p$, which implies $r(f) = s(e)$. This proves that $ef$ is a path in $G$. Finally, since $d_{\Gamma}(e^i) = d(e)$ and $d_\Gamma(f^j) = d(f)$ by Definition \ref{def: degreemaplimar}, then $ef$ is a two edge path in $G$ with the same degree.
\end{proof}

The following lemma proves \ref{KG2} for paths of length two with one edge of degree $B$.

\begin{lemma} \label{lemma: KG2AB}
Suppose $(\Lambda,d)$ is a $k$-graph and $(\Gamma,d_\Gamma)$ is the LiMaR-split of $\Lambda$ in degree $B$. If $e^if^j$ is a two edge two color path in $\Gamma$, where either $e^i$ or $f^j$ has degree $B$, then there exists a unique two color path $g^ih^l$  in $\Gamma$ where $d_\Gamma(g^i) = d_\Gamma(f^j)$, $d_\Gamma(h^l)=d_\Gamma(e^i)$, $s_{\Gamma}(h^l)=s_\Gamma(f^j)$, $r_{\Gamma}(g^i) = r_{\Gamma}(e^i)$, and $par(e^i)par(f^j) \sim par(g^i)par(h^l)$.
\end{lemma}

\begin{proof}
Let $e^if^j$ be a  two edge two-color path in $\Gamma$, where either $e^i$ or $f^j$ has degree $B$. Let the parent path of $e^if^j$ in $\Lambda$ be $ef$, which exists by Lemma \ref{lemma: parentpaths}. Since $\Lambda$ is a $k$-graph, there exists a unique two color path $gh$ such that $gh \sim ef$. In particular, we have that $r(g) = r(e)$. By Definition \ref{def: 1-skeletonlimarsplit}, there are $n(r(g))$ copies of $g$, $e$, and $r(g)$ in $\Gamma$. Moreover, if $1 \leq m \leq n(r(g))$, then $r_{\Gamma}(g^m) =r(g)^m = r(e)^m = r_{\Gamma}(e^m)$, which implies that there exists $g^i \in \Gamma$ such that $r_{\Gamma}(g^i) = r_{\Gamma}(e^i)$. 

By Remark \ref{rmk: distinctrange} and since $\ell(h) = s(g)$, there is exactly one copy of $h$ in $r_{\Gamma}^{-1}(s(g^i))$, which we will denote as $h^l$. In order to show $s_{\Gamma}(h^l) =  s_{\Gamma}(f^j)$, we will consider two cases.

For the first case, suppose that $d_{\Gamma}(h^l) \in A$ and $d_{\Gamma}(f^j) = B$. Since $d_{\Gamma}(h^l) \in A$ and $gh \sim ef$, then $d(g) = B$, which implies $g$ is in some partition of $\ell(h)$. Specifically, since $\ell(h)^l = r_{\Gamma}(h^l) = s_{\Gamma}(g^i)$, we know $g \in \mathcal{E}_l^{\ell(h)}$. Then, by Equation \eqref{sourcemap}, and since $gh \sim ef$, $s_{\Gamma}(h^l) = s(f)^p$, where $f \in \mathcal{E}_p^{s(f)}$.
Then, because $d(f) = B$ and $f \in \mathcal{E}_p^{s(f)}$, all copies of $f$ in $\Gamma$ are sourced from $s(f)^p$, including $f^j$. Thus, $s_{\Gamma}(h^l) = s_{\Gamma}(f^j)$.

For the second case, suppose $d_{\Gamma}(h^l) = B$ and $d_{\Gamma}(f^j) \in A$, which implies $d_{\Gamma}(e^i) = B$. From Equation \eqref{sourcemap}, we know that $s_{\Gamma}(h^l) = s(h)^q$ where $h \in \mathcal{E}_q^{s(h)}$. By hypothesis, $s_{\Gamma}(e^i) = r_{\Gamma}(f^j) = r(f)^j$, which implies $e \in \mathcal{E}_j^{r(f)}$. By Equation \eqref{sourcemap}, since $ef \sim gh$, then $s_{\Gamma}(f^j) = s(f)^q= s(h)^q = s_{\Gamma}(h^l)$. 

Finally, we know that $g^ih^l$ is the unique  path such that $par(e^i)par(f^j) \sim par(g^i)par(h^l)$, as Remark \ref{rmk: distinctrange} implies that there can be no other copy of $g$ that ranges to $r_{\Gamma}(e^i)$. Furthermore, since $d(e) = d(h)$ and $d(f) = d(g)$, then $d_\Gamma(e^i) = d_\Gamma(h^l)$ and $d_\Gamma(g^i) = d_\Gamma(f^j)$ by Definition \ref{def: degreemaplimar}. 
\end{proof}

Recall that in \cite[Theorem 5.4]{efgggp} it is proven that sink deletion on a $k$-graph, which is the removal of a vertex which is a sink in at least one degree, yields a $k$-graph whose $C^*$-algebra is Morita equivalent to the $C^*$-algebra of the original $k$-graph. Therefore, if $k \geq 3$, we can assume $\Lambda$ is degree $B$ sink-free in the following proof, which establishes \ref{KG2} for length two paths with edges whose degrees are both in $A$. For a review of the meaning of \textit{degree $B$ sink-free}, see Definition \ref{defn: d(B) sink}.

\begin{lemma} \label{lemma: KG2AA}
Suppose $(\Lambda,d)$ is a degree $B$ sink-free $k$-graph with $k \geq 3$, and suppose $(\Gamma, d_\Gamma)$ is the LiMaR-split of $\Lambda$ in degree $B$. If $g_1^je_1^n$ is a two edge two color path in $\Gamma$ with $d_\Gamma(g_1^j), d_\Gamma(e_1^n) \neq B$, then there exists a unique two color path $e_2^jg_2^i$ in $\Gamma$ such that $d_\Gamma(g_2^i) = d_\Gamma(g_1^j), d_\Gamma(e_2^j) =  d_\Gamma(e_1^n), s_\Gamma(g_2^i) = s_\Gamma(e_1^n)$, $r_\Gamma(e_2^j) = r_\Gamma(g_1^j)$, and $par(g_1^j)par(e_1^n) \sim par(e_2^j)par(g_2^i)$.

\end{lemma}

\begin{proof}
Let $g_1^je_1^n$ be a two color path in $\Gamma$  with $d_\Gamma(g_1^j) \neq d_\Gamma(e_1^n)$ and $d_\Gamma(g_1^j), d_\Gamma(e_1^n) \neq B$. Let $par(g_1^je_1^n) = g_1e_1$. Lemma \ref{lemma: parentpaths} implies $g_1e_1$ is a two color path in $\Lambda$. Let $e_2g_2$ be the  unique path in $\Lambda$ that commutes with  $g_1e_1$. In particular, $d(g_1) = d(g_2)$ and $d(e_1)=d(e_2)$. Furthermore, since $\Lambda$ is degree $B$ sink-free and $1\leq j \leq n(r(g_1))$ since $g_1^j$ exists, then $\mathcal{E}_j^{r(g_1)}$ is nonempty. Let $b_1$ denote the edge in $\mathcal{E}_j^{r(g_1)}$. Since $\Lambda$ is a $k$-graph, meaning it satisfies \ref{KG3}, $b_1g_1e_1$ commutes with five other three color paths, which we will use to prove there exists a unique two color path $e_2^jg_2^i$ in $\Gamma$ with $par(g_1^j)par(e_1^n) \sim par(e_2^j)par(g_2^i)$. To satisfy \ref{KG3} for the path $b_1g_1e_1$, we introduce the following edges in $\Lambda$: $b_2, b_3, b_4, e_3, e_4, g_3,$ and $g_4$, and set the following equivalence relations in $\Lambda$:

\begin{minipage}[r]{0.4\textwidth}
\begin{align*}
\ g_1e_1 \sim e_2g_2
&\implies b_1g_1e_1 \sim b_1e_2g_2 \\
&\,\,\,\,\text{let} \,\,\,\,\,\, b_1e_2 \sim e_3b_2 \\
&\implies  b_1e_2g_2 \sim e_3b_2g_2 \\
&\,\,\,\,\text{let} \,\,\,\,\,\, b_2g_2 \sim g_3b_3 \\
&\implies e_3b_2g_2 \sim e_3g_3b_3 \\
&\,\,\,\,\text{let} \,\,\,\,\,\, e_3g_3 \sim g_4e_4 \\
&\implies e_3g_3b_3 \sim g_4e_4b_3 \\
&\,\,\,\,\text{let} \,\,\,\,\,\, e_4b_3 \sim b_4e_1  \\
&\implies g_4e_4b_3 \sim g_4b_4e_1 \\
&\,\,\,\,\text{let} \,\,\,\,\,\, g_4b_4 \sim b_1g_1 \\
&\implies g_4b_4e_1 \sim b_1g_1e_1. \\
\end{align*}
\end{minipage}

First, we will establish that $e_2^j$ exists and $r_{\Gamma}(g_1^j) = r_{\Gamma}(e_2^j)$. Since $r(g_1) = r(e_2)$, by Definition \ref{def: 1-skeletonlimarsplit} the edge $e_2^j$ exists in $\Gamma$. Moreover, by Equation \eqref{rangemap}, $r_{\Gamma}(g_1^j) = r(g_1)^j = r(e_2)^j = r_{\Gamma}(e_2^j)$. 

Next, we must show that there exists a copy of $g_2$ that is concatenable with $e_2^j$. By the equivalencies above, we know there exists an edge $b_2$ such that $b_1e_2 \sim e_3b_2$. Let $b_2 \in \mathcal{E}_i^{s(b_2)}$ for some $i \in \{1, 2, \ldots, n(s(b_2))\}$. Then, by Equation \eqref{sourcemap}, $b_1e_2 \sim e_3b_2$ implies $s_{\Gamma}(e_2^j) = s(b_2)^i$. Additionally, since $b_2g_2$ is a path in $\Lambda$, there exists a $g_2^i \in \Gamma^1$ with $r_{\Gamma}(g_2^i) = s(b_2)^i = s_{\Gamma}(e_2^j)$ by Equation \eqref{rangemap}. Therefore $e_2^j g_2^i$ is a path in $\Gamma$.

By hypothesis, we know that $s_{\Gamma}(g_1^j) = r_{\Gamma}(e_1^n) = r(e_1)^n$. Then, since $b_1g_1 \sim g_4b_4$, $s_{\Gamma}(g_1^j) = s(b_4)^n$. By Equation \eqref{sourcemap}, this implies $b_4 \in \mathcal{E}_n^{s(b_4)}$.

Lastly we show that $s_{\Gamma}(e_1^n) = s_{\Gamma}(g_2^i)$. Suppose $b_3 \in \mathcal{E}_h^{s(b_3)}$ for some $h \in \{1,2,\ldots, n(s(b_3))\}$. Then by Equation \eqref{sourcemap}, since $b_2g_2 \sim g_3b_3$, we have $s_{\Gamma}(g_2^i) = s(b_3)^h$. Next, because $b_4e_1 \sim e_4b_3$ and $b_4 \in \mathcal{E}_n^{s(b_4)}$, we know $s_{\Gamma}(e_1^n) = s(b_3)^h = s_{\Gamma}(g_2^i)$ by Equation \eqref{sourcemap}.

Putting it all together, we have  established that $e_2^jg_2^i$ is a two color path in $\Gamma$ such that $d_\Gamma(g_2^i) = d(g_2) = d(g_1) = d_\Gamma(g_1^j), d_\Gamma(e_2^j) = d(e_2) = d(e_1) = d_\Gamma(e_1^n), s_{\Gamma}(g_2^i) = s_{\Gamma}(e_1^n)$, $r_{\Gamma}(e_2^j) = r_{\Gamma}(g_1^j)$, and $par(g_1^j)par(e_1^n) \sim par(e_2^j)par(g_2^i)$. Furthermore, there cannot exist another such path since Remark \ref{rmk: distinctrange} implies that $e_2^j$ is the only copy of $e_2$ with the same range as $g_1^j$.
\end{proof}

The following lemma establishes that \ref{KG3} is satisfied after LiMaR-splitting. 

\begin{lemma} \label{lemma: KG3}
Suppose $(\Lambda,d)$ is a $k$-graph with $k \geq 3$ and $(\Gamma,d_\Gamma)$ is the LiMaR-split of $\Lambda$ in degree $B$. All three color paths of length three in $\Gamma$ commute with exactly five other paths of length three that have the same range and source and a different color order.

\end{lemma}

\begin{proof}
For all $n,m \in \NN$, let $e_n, e_m, b_n, b_m, g_n, g_m$ in $\Lambda^1$ be such that $d(e_n) = d(e_m)$, $d(b_n) = d(b_m)$, $d(g_n) = d(g_m)$, $d(e_n) \neq d(b_n)$, $d(e_n) \neq d(g_n)$, and $d(b_n) \neq d(g_n)$. Suppose that $e_1^jg_1^kb_1^l$ is a path of length three in $\Gamma$ that commutes with two potentially different paths of length three of the same color order: $b_2^ug_2^ve_2^w$ and $b_3^xg_3^ye_3^z$. By Lemma \ref{lemma: parentpaths}, the parent of any path of length three in $\Gamma$ is a path of length three in $\Lambda$, which implies that $par(e_1^jg_1^kb_1^l) = e_1g_1b_1$ is a path of length three in $\Lambda$. Since all relations of paths of length three can be determined via swapping commuting paths of length two, the parents of commuting paths of length three in $\Gamma$ must also commute in $\Lambda$. Since \ref{KG3} is true for $\Lambda$, we know that $par(b_2^ug_2^ve_2^w) = par(b_3^xg_3^ye_3^z)$, which allows us to rename $b_3^xg_3^ye_3^z$ to $b_2^xg_2^ye_2^z$. Then, observe that $b_2^u$ and $b_2^x$ must have the same range since they are both the last edge in a path of length three that commutes with $e_1^jg_1^kb_1^l$. However, by Remark \ref{rmk: distinctrange}, copied edges must have different ranges, so we must have $b_2^u = b_2^x$, which implies that $r(g_2^v) = r(g_2^y) = s(b_2^u)$, which means we have $g_2^v = g_2^y$. By similar reasoning, we must have $e_2^w = e_2^z$. Thus, we have $b_2^ug_2^ve_2^w = b_3^xg_3^ye_3^z$, which implies that \ref{KG3} is satisfied for $\Gamma$.
\end{proof}
 
\begin{thm} \label{thm:kgraph}
    Let $(\Lambda,d)$ be a source-free, row-finite $k$-graph. If $k \geq 3$, additionally assume that $(\Lambda,d)$ is degree $B$ sink-free. Then the LiMaR-split $(\Gamma,d_\Gamma)$ of $\Lambda$ in degree $B$ is a source-free, row-finite $k$-graph.
\end{thm}
\begin{proof}
By Theorem 2.1. in \cite{efgggp}, $\Gamma$ is a $k$-graph if \ref{KG0}-\ref{KG3} are satisfied. By construction, \ref{KG0} and \ref{KG1} are satisfied.  By Lemmas \ref{lemma: KG2AB}, \ref{lemma: KG2AA}, and \ref{lemma: KG3}, we have that \ref{KG2} and \ref{KG3} are satisfied. Thus, $\Gamma$ is a $k$-graph.

Let $v^j$ be a vertex in $\Gamma$ with $par(v^j) = v$, and fix $\mathbf{e}_i \in \mathbb{N}^k$. Since $\Lambda$ is source-free, there exists an $f \in \Lambda^{\mathbf{e}_i}$ with $r(f) = v$. Then, by Definition \ref{def: 1-skeletonlimarsplit}, there exists $f^j \in \Gamma^{\mathbf{e}_i}$ with $r(f^j) = v^j$. Thus, $\Gamma$ is source-free.

Now, to prove $\Gamma$ is row-finite, observe that $f^j$ is the only copy of $f$ whose range is $v^j$, and for every edge in $\Lambda$ whose range is $v$, there is exactly one copy of the edge in $\Gamma$ whose range is $v^j$. Therefore $|v ^j\Gamma^{\mathbf{e}_i}| = |v\Lambda^{\mathbf{e}_i}| <\infty$ because $\Lambda$ is row-finite.
\end{proof}

\section{Proving Morita equivalence of $C^*(\Lambda)$ and $C^*(\Gamma)$} \label{sec: ME}

The final section explores the relationship between the $C^*$-algebras $C^*(\Lambda)$ and $C^*(\Gamma)$ as well as the relationship between the Kumjian-Pask algebras $\mathrm{KP}_\mathbb{C}(\Lambda)$ and $\mathrm{KP}_\mathbb{C}(\Gamma)$. The results of this section culminate in the proof that the $C^*$-algebras $C^*(\Lambda)$ and $C^*(\Gamma)$ are Morita equivalent, where $\Gamma$ is the result of LiMaR-splitting a $k$-graph $\Lambda$ that is degree $B$ sink-free if $k \geq 3$.

The $j^{th}$ copy of an edge is presented in Definition \ref{def: 1-skeletonlimarsplit}. Now we define the $j^{th}$ copy of a path. For each $f= f_m\cdots f_1 \in \Lambda^{\neq \mathbf{0}}$ and for each $1 \leq j \leq n(r(f))$, we define $f^j \in \Gamma$ by $f^j = f_m^{a_m} f_{m-1}^{a_{m-1}} \cdots f_1^{a_1}$, where $a_m = j$ and $s(f_i)^{a_{i-1}}= s_\Gamma(f_{i}^{a_{i}})$.  As the $s_\Gamma(e^i) = s(e)^j$ for some $1 \leq j \leq n(s(e))$, $1 \leq a_{i-1} \leq n( s(f_i)) = n( r(f_{i-1}))$.  Thus, $f_m^{a_m} f_{m-1}^{a_{m-1}} \cdots f_1^{a_1}$ is a path in $\Gamma$.

Let $S \subseteq \{1, \ldots, k \}$.  Denote by $\mathbbm{1}_S = \sum_{ i \in S} \mathbf{e}_i$.  When $S = \{ 1, \ldots, k \}$, we write $\mathbbm{1}$ instead of $\mathbbm{1}_{\{1,\ldots, k\} }$. The next three lemmas are technical lemmas about the $j^{th}$ copy of a path in $\Lambda$. These results free us from requiring $\Lambda$ to satisfy a pairing condition like was done by Listhartke, for example [c.f. \cite{listhartke}]. 

\begin{lem}\label{lem-outsplit-uniq-source}
Let $\Lambda$ be a $k$-graph that is degree $B$ sink-free if $k \geq 3$, let $\Gamma$ be the LiMaR-split of $\Lambda$ at $w$ of degree $B$, and let $S \subseteq \{1, \ldots, k\}$ such that $B \in \bigcup_{ i \in S } \{ \mathbf{e}_i \}$.   Let $f , g \in \Lambda^{\mathbbm{1}_S}$ such that $f = g$.  Then for all $1 \leq j \leq n(r(f))$, $f^j =g^j$ and $s_\Gamma(f^j) = s_\Gamma(f^1)$. 
\end{lem}

\begin{proof}
    First we prove that if $S \subseteq \{1,2, \ldots,k\}$ and $f, g \in \Lambda^{\mathbbm{1}_S}$ such that $f=g$, then $f^j = g^j$ for all $1 \leq j \leq n(r(f))$. If $|S| =1$, then $f^j = g^j$. Now suppose the claim is true for all $S$ with $|S| = m-1$ for some $1 < m \leq k$. Suppose $|S|=m$ and $f,g \in \Lambda^{\mathbbm{1}_S}$ such that $f=g$ and $1 \leq j \leq n(r(f))$. Fix $i \in S$ and write $f = \alpha f'$ where $d(\alpha) = e_i$. Then, since $f=g$, we may write $g = \alpha g'$ where $g' = f'$. Suppose $s_\Gamma(\alpha^j) = s(\alpha)^\ell$. Then $f^j = \alpha^j f'^\ell$ and $g^j = \alpha^j g'^\ell$. Since $f'=g' \in \Lambda^{\mathbbm{1}_{S\setminus i}}$, then $f'^\ell = g'^\ell$ by induction hypothesis. So $f^j = g^j$.

    Now suppose $S \subseteq \{1,\ldots, k\}$ such that $B \in \cup_{i \in S}\{\mathbf{e}_i\}$ and $f \in \Lambda^{\mathbbm{1}_S}$. Then we can write $f= \alpha f'$ where $d(\alpha) = B$. Suppose $s_\Gamma(\alpha^j) = s(\alpha)^\ell$. Then $s_\Gamma(\alpha^1) = s(\alpha)^\ell$ as well since $d(\alpha) = B$. Then $f^j = \alpha^j f'^\ell$ and $f^1 = \alpha^1 f'^\ell$, so $s_\Gamma(f^j) = s_\Gamma(f'^\ell) = s(f^1)$.
\end{proof}

\begin{lem}\label{lem-uniq1}
Let $\Lambda$ be a $k$-graph that is degree $B$ sink-free if $k \geq 3$, and let $\Gamma$ be the LiMaR-split of $\Lambda$ at $w$ of degree $B$.  Let $f = f_m\cdots f_1 \in \Lambda$.  Suppose $s_\Gamma(f_m^j) = s_\Gamma(f_m^1) = s(f_m)^i$.  Then $f^j = f_m^j f_{m-1}^{a_{m-1}}\cdots f_{1}^{a_1}$ and $f^1 = f_m^1 f_{m-1}^{a_{m-1}}\cdots f_{1}^{a_1}$, where $s(f_i)^{a_{i-1}}= s_\Gamma(f_{i}^{a_{i}})$.  
\end{lem}

\begin{proof}
Set $f^j = f_m^j f_{m-1}^{a_{m-1}} \cdots f_{1}^{a_1}$ and $f^1 = f_m^1 f_{m-1}^{b_{m-1}} \cdots f_1^{b_1}$, where $s(f_i)^{a_{i-1}}= s_\Gamma(f_{i}^{a_{i}})$ and $s(f_i)^{b_{i-1}}= s_\Gamma(f_{i}^{b_{i}})$.  Since $s_\Gamma(f_m^j) = s_\Gamma(f_m^1) = s(f_m)^i$, we have $a_{m-1} = b_{m-1} = i$.  Therefore,
\[
s(f_{m-1})^{a_{m-2}} = s_\Gamma( f_{m-1}^{a_{m-1}} ) = s_\Gamma( f_{m-1}^{b_{m-1}} ) = s(f_{m-1})^{b_{m-2}}.
\]
We can continue this process to show that $a_i = b_i$ for all $i$.
\end{proof}

\begin{lem}\label{lem-uniq2}
Let $\Lambda$ be a $k$-graph that is degree $B$ sink-free if $k \geq 3$, and let $\Gamma$ be the LiMaR-split of $\Lambda$ at $w$ of degree $B$.  Let $\mu \in \Gamma^{\mathbbm{1}}$ and let $f = \operatorname{par} (\mu)$.  Then $f^j = \mu$ where $r_\Gamma(\mu) = r(f)^j$.
\end{lem}

\begin{proof}
Write $\mu = g_k \cdots g_1$, where $\deg(g_1)=B$.  Then $f = f_k\cdots f_1$, where $f_i = \operatorname{par}(g_i)$.   Then $f^j = f_k^j f_{k-1}^{a_{k-1}} \cdots f_{1}^{a_1}$.  Suppose $f_1 \in \mathcal{E}_\ell^{s(f_1)}$.  Then $s_\Gamma(g_1) =s(f_1)^\ell$ which implies that
\[
s_\Gamma( \mu) = s(f_1)^{\ell} = s_\Gamma(f^j ).
\]
By construction $r_\Gamma(\mu) = r(f)^j = r_\Gamma(f^j)$ and $\operatorname{par}(\mu) = \operatorname{par}( f^j)$, we have $\mu = f^j$.
\end{proof}

The following lemma, which we will use to prove LiMaR-splitting preserves Morita equivalence, was proven for Listhartke's version of outsplitting \cite[Lemma 4.0.1.1]{listhartke}. It holds in our more general setting as well.

\begin{lem}\label{lem-ck2-1}
Let $\Lambda$ be a $k$-graph that is degree $B$ sink-free if $k \geq 3$, and let $\Gamma$ be the LiMaR-split of $\Lambda$ at $w$ of degree $B$.  Then for all $x \in \Lambda^1$ and for all $1 \leq j \leq n(r(x))$ such that $s_\Gamma (x^j) = s_\Gamma(x^1)$,
\[
\sum_{ f \in r(x) \Lambda^{\mathbbm{1}} } t_{f^j} t_{(f^1)^*} t_{x^1} = t_{x^j}
\]
and
\[\sum_{ f \in r(x) \Lambda^{\mathbbm{1}} } t_{(x^1)^*} t_{f^1} t_{(f^j)^*} = t_{(x^j)^*} \]
in $\mathrm{KP}_\mathbb{F}(\Gamma)$.
\end{lem}

\begin{proof}
Let $x \in \Lambda^1$ and let $1 \leq j \leq n(r(x))$.  We claim 
\[ 
\sum_{ f \in r(x) \Lambda^{\mathbbm{1}} } t_{f^j} t_{(f^1)^*} t_{x^1} = \sum_{ \mu \in s(x^1) \Gamma^{\mathbbm{1} - d(x)} } t_{x^j} t_{\mu} t_{\mu^*} t_{(x^1)^*} t_{x^1}. 
\]
Let $f \in \Lambda^{\mathbbm{1}}$ and write $f = f_k\cdots f_1$ with $\deg(f_k)=\deg(x)$.  Suppose $t_{(f^1)^*}t_{x^1} \neq 0$.  Then $f^1 = x^1 \mu$ for some $\mu \in s_\Gamma(x^1)\Gamma^{\mathbbm{1}-\deg(x)}$.  Since $s_\Gamma( x^j) = s_\Gamma(x^1)$, by Lemma~\ref{lem-uniq1}, $f^j = x^j \mu$. 

Suppose $\mu \in s_\Gamma(x^1)\Gamma^{\mathbbm{1}-\deg(x)}$.  Note that $x^1 \mu \in r_\Gamma(x^1) \Gamma^{\mathbbm{1}}$ and $x^j \mu  \in r_\Gamma(x^j) \Gamma^{\mathbbm{1}}$ since $s_\Gamma( x^j) = s_\Gamma(x^1)$.  Let $f = \operatorname{par}(x^1 \mu) = \operatorname{par}(x^j \mu)$.  By Lemma~\ref{lem-uniq2}, \[
f^j = x^j \mu
\]
and $t_{(f^1)^*} t_{x^1} \neq 0$.  Thus, proving the claim.

Using the above claim, 
\begin{align*}
\sum_{ f \in r(x) \Lambda^{\mathbbm{1}} } t_{f^j} t_{(f^1)^*} t_{x^1}  &= \sum_{ \mu \in s(x^1) \Gamma^{\mathbbm{1} - \deg(x)} } t_{x^j} t_\mu t_{\mu^*} t_{(x^1)^*}t_{x^1} \\
	&= \sum_{ \mu \in s(x^1) \Gamma^{\mathbbm{1} - \deg(x)}} t_{x^j} t_\mu t_\mu^* p_{ s_\Gamma(x^1) } \\
	&= \sum_{ \mu \in s(x^1) \Gamma^{\mathbbm{1} - \deg(x)}} t_{x^j} t_\mu t_\mu^* \\
	&= t_{x^j}.
\end{align*}
This proves the first equality. The proof of the second is similar.
\end{proof}

\begin{lem}\label{lem-ck2-2}
Let $\Lambda$ be a $k$-graph that is degree $B$ sink-free if $k \geq 3$, and let $\Gamma$ be the LiMaR-split of $\Lambda$ at $w$ of degree $B$.  Let $x \in \Lambda^1$.  Fix $1 \leq j \leq n(r(x))$.  Suppose $s_\Gamma(x^1)= s(x)^m$ and $s_\Gamma(x^j) = s(x)^n$.  Then 
\[
\left( \sum_{ f \in r(x)  \Lambda^{\mathbbm{1}}} t_{f^j} t_{(f^1)^*} \right) t_{x^1}\left( \sum_{ f \in s(x) \Lambda^{\mathbbm{1}} }  t_{f^m} t_{(f^1)^*} \right)  =  t_{x^j}  \sum_{ f \in s(x)  \Lambda^{\mathbbm{1}}}t_{f^n} t_{(f^1)^*}.
\]
in $\mathrm{KP}_\mathbb{F}(\Gamma)$.
\end{lem}

\begin{proof}
Suppose $\deg(x)=B$.  Then $m= n$ and by Lemma~\ref{lem-ck2-1}, 
\begin{align*}
\left( \sum_{ f \in r(x)  \Lambda^{\mathbbm{1}}} t_{f^j} t_{(f^1)^*} \right) t_{x^1}\left( \sum_{ f \in s(x) \Lambda^{\mathbbm{1}} }  t_{f^m} t_{(f^1)^*} \right)  &=  t_{x^j}  \sum_{ f \in s(x)  \Lambda^{\mathbbm{1}}}t_{f^m} t_{(f^1)^*}\\
&= t_{x^j}  \sum_{ f \in s(x)  \Lambda^{\mathbbm{1}}}t_{f^n} t_{(f^1)^*}.
\end{align*}

Suppose $\deg(x)\neq B$.  
Suppose $t_{(f^1)^*}t_{x^1} \neq 0$.  Then $f = xh$ for some $h=h_{k-1} \cdots h_1 \in s(x)\Lambda^{\mathbbm{1}-\deg(x)}$.  Note $f^1 = x^1 h^m$ and $f^j = x^j h^n$.  By Lemma~\ref{lem-outsplit-uniq-source}, for all $h \in \Lambda^{\mathbbm{1}-\deg(x)}$, there exists a positive integer $\ell(h)$ such that $s_{\Gamma}(h^x) = s_{\Gamma}(h^y)= s(h)^{\ell(h)}$ for all $x,y$.  Therefore,
\begin{align*}
&\left( \sum_{ f \in r(x) \Lambda^{\mathbbm{1}}}  t_{f^j} t_{(f^1)^*} \right) t_{x^1} \left( \sum_{ f \in s(x) \Lambda^{\mathbbm{1}} }  t_{f^m} t_{(f^1)^*} \right) \\
&= \sum_{ h \in s(x)\Lambda^{\mathbbm{1}-\deg(x)} } t_{x^j} t_{h^n } t_{(h^m)^*}\left( \sum_{ f \in s(x) \Lambda^{\mathbbm{1}} }  t_{f^m} t_{(f^1)^*} \right).
\end{align*}

Suppose $h=h_{k-1} \cdots h_1 \in s(x)\Lambda^{\mathbbm{1}-\deg(x)}$ and $t_{h^m}^* t_{f^m} \neq 0$.  Write $f = f_k \cdots f_1$ where $\deg(f_i) = \deg(h_{i-1})$ and $\deg(f_1)=\deg(x)$.  Then $f^m= (f_k\cdots f_2)^m f_1^{a_1}$.  Since $t_{h^m}^* t_{f^m} \neq 0$ we must have $(f_k\cdots f_2)^m = h^m$ which implies $\operatorname{par}( (f_k\cdots f_2)^m ) = \operatorname{par}(h^m)=h$ and $s(h)^{a_1} = s(f_k\cdots f_2)^{a_1} = s_\Gamma( (f_k\cdots f_2)^m ) ) = s_\Gamma(h^m)= s(h)^{\ell(h)}$.  Thus, $f = h f_1$ and $a_1 = \ell(h)$.  Hence, $h^n f_1^{\ell(h)} \in \Gamma^{\mathbbm{1}}$ and by Lemma~\ref{lem-uniq2}, $f^n = h^n f_1^{\ell(h)}$.  

Using the above observations, we have
\begin{align*}
&\left( \sum_{ f \in r(x) \Lambda^{\mathbbm{1}}}  t_{f^j} t_{(f^1)^*} \right) t_{x^1} \left( \sum_{ f \in s(x) \Lambda^{\mathbbm{1}} }  t_{f^m} t_{(f^1)^*} \right) \\
&= \sum_{ h \in s(x)\Lambda^{\mathbbm{1}-\deg(x)} }  t_{x^j} t_{h^n } t_{(h^m)^*}\left( \sum_{ f \in s(x) \Lambda^{\mathbbm{1}} }  t_{f^m} t_{(f^1)^*} \right) \\
&= \sum_{  h \in s(x)\Lambda^{\mathbbm{1}-\deg(x)} }  \sum_{c \in s(h) \Lambda^{\deg(x)} }  t_{x^j} t_{h^n } t_{c^{\ell(h)}} t_{ ((hc)^1)^*}   \\
&= \sum_{  h \in s(x)\Lambda^{\mathbbm{1}-\deg(x)}  }  \sum_{c \in s(h) \Lambda^{\deg(x)} }  t_{x^j} t_{(hc)^n}  t_{ ((hc)^1)^*} \\
&= t_{x^j} \sum_{ f \in s(x) \Lambda^{\mathbbm{1}} }  t_{f^n} t_{(f^1)^*}\qedhere
\end{align*}
\end{proof}

\begin{lem}\label{lem-paths}
Let $\Lambda$ be a $k$-graph that is degree $B$ sink-free if $k \geq 3$, and let $\Gamma$ be the LiMaR-split of $\Lambda$ at $w$ of degree $B$.  For all $ e\in \Lambda^1$, set $T_e =  \sum_{ f \in s(e)\Lambda^{\mathbbm{1}}} t_{e^1} t_{ f^j } t_{(f^1)^*}$, where $s_\Gamma(e^1) = s(e)^j$.  If $h= h_m\cdots h_1 \in \Lambda$, then 
\[
T_{h_m} \cdots T_{h_1}  = t_{ h^1} \sum_{ f \in s(h) \Lambda^{\mathbbm{1}}}  t_{f^\ell} t_{(f^1)^*}
\]
in $\mathrm{KP}_\mathbb{F}(\Gamma)$, where $s_\Gamma(h^1) = s(h)^\ell$.
\end{lem}

\begin{proof}
We prove this by induction.  If the length of $h$ is $1$, this is clear from the definition of $T_e$.  Assume the statement of the lemma is true for all paths of length $m$.  Let $h = h_{m+1} \cdots h_1= ge$, where $g$ has length $m$ and $e$ has length $1$.  Then 
\[
T_{h_{m+1}} \cdots T_{h_1} = t_{g^1} \left(\sum_{ f \in s(g) \Lambda^{\mathbbm{1}}}  t_{f^\ell} t_{(f^1)^*}\right)t_{e^1}  \left( \sum_{ f \in s(e)\Lambda^{\mathbbm{1}}} t_{ f^j } t_{(f^1)^*} \right)
\]
where $s(g)^\ell= s_\Gamma(g^1)$ and $s(e)^j= s_\Gamma(e^1)$.  Note that $h^1 = g^1 e^\ell$.  Let $n$ be a positive integer such that $s(e)^n = s_\Gamma(e^\ell)$.  By Lemma~\ref{lem-ck2-2}, 
\[
\left(\sum_{ f \in s(g) \Lambda^{\mathbbm{1}}}  t_{f^\ell} t_{(f^1)^*}\right)t_{e^1}  \left( \sum_{ f \in s(e)\Lambda^{\mathbbm{1}}} t_{ f^j } t_{(f^1)^*} \right) = t_{e^{\ell}} \sum_{ f \in s(e) \Lambda^{\mathbbm{1}} } t_{f^n} t_{(f^1)^*}.
\] 
Hence,
\begin{align*}
T_{h_{m+1}} \cdots T_{h_1} &= t_{g^1} \left(\sum_{ f \in s(g) \Lambda^{\mathbbm{1}}}  t_{f^\ell} t_{(f^1)^*}\right)t_{e^1}  \left( \sum_{ f \in s(e)\Lambda^{\mathbbm{1}}} t_{ f^j } t_{(f^1)^*} \right) \\
&= t_{g^1}t_{e^{\ell}} \sum_{ f \in s(e) \Lambda^{\mathbbm{1}} } t_{f^n} t_{(f^1)^*} \\
	&= t_{(ge)^1} \sum_{ f \in s(e) \Lambda^{\mathbbm{1}} } t_{f^n} t_{(f^1)^*} \\
	&= t_{h^1} \sum_{ f \in s(e) \Lambda^{\mathbbm{1}} } t_{f^n} t_{(f^1)^*}
\end{align*}
and $s_{\Gamma}(h^1) = s_\Gamma( (ge)^1) = s_\Gamma(e^\ell) = s(e)^n$.
\end{proof}

\begin{lem}\label{lem-paths-2}
Let $\Lambda$ be a $k$-graph that is degree $B$ sink-free if $k \geq 3$, and let $\Gamma$ be the LiMaR-split of $\Lambda$ at $w$ of degree $B$.   Let $v \in \Lambda^0$ and let $1 \leq j \leq n(v)$.  Then 
\[
\{ g \in \Gamma^{\mathbbm{1}} : r_\Gamma(g)  = v^j \} = \{ f^j : f \in v\Lambda^{\mathbbm{1}} \}.
\] 
\end{lem}

\begin{proof}
For $f \in v\Lambda^{\mathbbm{1}}$, $f^j \in \Gamma^{\mathbbm{1}}$ such that $r_\Gamma(f^j) = r(f)^j = v^j$.  Suppose $g \in \Gamma^{\mathbbm{1}}$ such that $r_\Gamma(g)  = v^j$.  Set $f = \operatorname{par} (g)$.  By Lemma~\ref{lem-uniq2}, $f \in v \Lambda^{\mathbbm{1}}$ such that $f^j = g$.  
\end{proof}

\begin{thm}\label{thm-KP}
Let $\Lambda$ be a source-free, degree $B$ sink-free if $k \geq 3$, row-finite $k$-graph, and let $\Gamma$ be the LiMaR-split of $\Lambda$ at $w$ of degree $B$. Let $\{p_v,s_e : v \in \Lambda^0, e \in \Lambda^1\}$ and $\{q_{v^i},t_{e^j} : v^i \in \Gamma^0, e^j \in \Gamma^1\}$ denote the universal Kumijan-Pask $\Lambda$- and $\Gamma$-families, respectively. There exists an injective $\mathbb{Z}^k$-graded homomorphism from $\psi \colon \mathrm{KP}_\mathbb{F}(\Lambda) \to \mathrm{KP}_\mathbb{F}(\Gamma)$ such that $\psi(p_v)= q_{v^1}$, $\psi( s_e ) = \sum_{ f \in s(e)\Lambda^{\mathbbm{1}}} t_{e^1} t_{ f^j } t_{(f^1)^*}$, and $\psi( s_{e^*} ) = \sum_{ f \in s(e)\Lambda^{\mathbbm{1}}}   t_{f^1}t_{ (f^j)^* }t_{(e^1)^*}$, where $s_\Gamma(e^1)= s(e)^j$.
\end{thm}

\begin{proof}
Set $P_{v} = q_{v^1}$, $S_e =  \sum_{ f \in s(e)\Lambda^{\mathbbm{1}}} t_{e^1} t_{ f^j } t_{(f^1)^*}$, and $S_{e^*}=\sum_{ f \in s(e)\Lambda^{\mathbbm{1}}} t_{f^1}t_{ (f^j)^* }t_{(e^1)^*}$, where $s_\Gamma(e^1)= s(e)^j$. To get $\psi$, by Lemma~\ref{lem-KPprime}, it is enough to show that $\{Q_v, T_e : v \in \Lambda^0, e \in \Lambda^1 \}$ satisfies \ref{KP1'}, \ref{KP2'}, \ref{KP3'}, \ref{KP4'}.  Since $\{ Q_v = q_{v^1} : v \in \Lambda^0 \}$ is a sub-collection of the vertex projections in $\mathrm{KP}_\mathbb{F}(\Gamma)$, they are mutually orthogonal idempotents.  Hence, \ref{KP1'} holds.

Next, we prove \ref{KP2'}.  Let $a,x,b,y \in \Lambda^1$.  Then 
\begin{align*}
P_{r(a)} S_a &= \sum_{ f \in s(a)\Lambda^{\mathbbm{1}}} q_{r(a)^1} t_{a^1} t_{ f^j } t_{(f^1)^*} = \sum_{ f \in s(a)\Lambda^{\mathbbm{1}}} q_{r_\Gamma(a^1)} t_{a^1} t_{ f^j } t_{(f^1)^*} = T_a \\
S_a P_{s(a)} &= \sum_{ f \in s(a)\Lambda^{\mathbbm{1}}}  t_{a^1} t_{ f^j } t_{(f^1)^*} q_{s(a)^1} = \sum_{ f \in s(a)\Lambda^{\mathbbm{1}}}  t_{a^1} t_{ f^j } t_{(f^1)^*} q_{r_{\Gamma}(f^1)} = T_a \\
P_{s(a)} S_{a^*} &= \sum_{ f \in s(a)\Lambda^{\mathbbm{1}}}   q_{s(a)^1}t_{f^1}t_{ (f^j)^* }t_{(a^1)^*} = \sum_{ f \in s(a)\Lambda^{\mathbbm{1}}}   P_{r_\Gamma(f^1)}t_{f^1}t_{ (f^j)^* }t_{(a^1)^*} =T_{a^*} \\
S_{a^*} P_{r(a)} &= \sum_{ f \in s(a)\Lambda^{\mathbbm{1}}}   t_{f^1}t_{ (f^j)^* }t_{(a^1)^*}  q_{r(a)^1} = \sum_{ f \in s(a)\Lambda^{\mathbbm{1}}}   t_{f^1}t_{ (f^j)^* }t_{(a^1)^*}  q_{r_\Gamma(a^1)} = T_{a^*}.
\end{align*}
Assume $ax \sim by$.  Let $s_\Gamma(a^1)= s(a)^j$ and $s_\Gamma(b^1) =s(b)^\ell$.  Then $(ax)^1 = a^1 x^j$ and $(by)^1=b^1y^\ell$ are paths in $G_\Gamma$.  By Lemma~\ref{lemma: KG2AB}, Lemma~\ref{lemma: KG2AA}, and Definition~\ref{def: 1-skeletonlimarsplit}, $a^1 x^j \sim b^1y^\ell$.  Let $i, m, n$ be positive integers such that $s_\Gamma(x^1) = s(x)^m$, $s_\Gamma(x^j) =s(x)^n$, $s_\Gamma(y^1)=s(y)^i$.  Since $a^1 x^j \sim b^1y^\ell$, $s_\Gamma(y^{\ell}) = s_\Gamma(x^j) =s(x)^n$.  Then \begin{equation*}
S_a S_x =t_{a^1} t_{x^j} \sum_{ f \in s(x) \Lambda^{\mathbbm{1}}} t_{f^n}t_{(f^1)^*} 
	= t_{b^1} t_{y^\ell } \sum_{ f \in s(y) \Lambda^{\mathbbm{1}}} t_{f^n}t_{(f^1)^*} 
	= S_{b} S_y, 
\end{equation*}
where the 1st and 3rd equalities follow from Lemma~\ref{lem-paths}.  Hence, \ref{KP2'} holds.

For \ref{KP3'}, let $a, b \in \Lambda^1$.  Then 
\begin{align*}
S_{a^*}S_a &= \left( \sum_{ f \in s(a)\Lambda^{\mathbbm{1}}} t_{f^1} t_{ (f^j)^* } t_{(a^1)^*} \right)\left( \sum_{ f \in s(a)\Lambda^{\mathbbm{1}}} t_{a^1} t_{ f^j } t_{(f^1)^*}\right) \\
&= \sum_{ f , g\in s(a)\Lambda^{\mathbbm{1}}}  t_{ f^1 } t_{(f^j)^*}t_{(a^1)^*} t_{a^1} t_{ g^j } t_{(g^1)^*} = \sum_{ f , g\in s(a)\Lambda^{\mathbbm{1}}} t_{ f^1 } t_{(f^j)^*}q_{ s_\Gamma(a^1) }  t_{ g^j } t_{(g^1)^*} \\
&= \sum_{ f , g\in s(a)\Lambda^{\mathbbm{1}}}  t_{ f^1 } t_{(f^j)^*}q_{ r_\Gamma(g^j) }  t_{ g^j } t_{(g^1)^*} = \sum_{ f , g\in s(a)\Lambda^{\mathbbm{1}}}  t_{ f^1 } t_{(f^j)^*}t_{ g^j } t_{(g^1)^*} \\
&= \sum_{ f \in s(a)\Lambda^{\mathbbm{1}}}  t_{ f^1} q_{s_{\Gamma(f^j)} }  t_{(f^1)^*} \\
		&\qquad \text{by Lemma~\ref{lem-outsplit-uniq-source}} \\
		&=  \sum_{ f \in s(a)\Lambda^{\mathbbm{1}}}  t_{ f^1 } q_{s_{\Gamma(f^1)} }  t_{(f^1)^*}  =  \sum_{ f \in s(a)\Lambda^{\mathbbm{1}}}  t_{ f^1}  t_{(f^1)^*}.
\end{align*} 
By Lemma~\ref{lem-paths-2},
\[
S_{a^*}S_a = \sum_{ f \in s(a)\Lambda^{\mathbbm{1}}}  t_{ f^1}  t_{(f^1)^*} = \sum_{ g \in s(a)^1 \Gamma^{\mathbbm{1}}} t_{g} t_{g^*} = q_{s(a)^1 } = P_{s(a)}.
\]
Moreover,
\begin{align*}
S_{a^*}S_b &= \left( \sum_{ f \in s(a)\Lambda^{\mathbbm{1}}} t_{f^1} t_{ (f^j)^* } t_{(a^1)^*} \right)\left( \sum_{ f \in s(b)\Lambda^{\mathbbm{1}}} t_{b^1} t_{ f^\ell } t_{(f^1)^*}\right) \\
&= \left( \sum_{ f \in s(a)\Lambda^{\mathbbm{1}}} t_{f^1} t_{ (f^j)^* }  \right)t_{(a^1)^*} t_{b^1}\left( \sum_{ f \in s(b)\Lambda^{\mathbbm{1}}}  t_{ f^\ell } t_{(f^1)^*}\right) = 0
\end{align*}
since $a \neq b$ which implies $a^1\neq b^1$.  Thus, showing that \ref{KP3'} holds.

Lastly, we prove \ref{KP4'}.  Let $v \in \Lambda^0$ and let $1 \leq i \leq k$.  Then 
\begin{align*}
\sum_{ e \in v \Lambda^{\mathbf{e}_i } } S_e S_{e^*} &= \sum_{ e\in v \Lambda^{\mathbf{e}_i } }  \left( \sum_{ f \in s(e)\Lambda^{\mathbbm{1}}} t_{e^1} t_{ f^j } t_{(f^1)^*}\right)\left( \sum_{ f \in s(e)\Lambda^{\mathbbm{1}}} t_{f^1}  t_{ (f^j)^* } t_{(e^1)^*} \right) \\
&= \sum_{ e\in v \Lambda^{\mathbf{e}_i } } t_{e^1} \left( \sum_{ f \in s(e) \Lambda^{\mathbbm{1}} } t_{f^j} t_{(f^1)^*} t_{f^1} t_{(f^j)^*} \right) t_{(e^1)^*} \\
&= \sum_{ e\in v \Lambda^{\mathbf{e}_i } } t_{e^1} \left( \sum_{ f \in s(e) \Lambda^{\mathbbm{1}} } t_{f^j}  q_{s_{\Gamma(f^1)}}  t_{(f^j)^*} \right) t_{(e^1)^*}  \\
&=  \sum_{ e\in v \Lambda^{\mathbf{e}_i } } t_{e^1} \left( \sum_{ f \in s(e) \Lambda^{\mathbbm{1}} } t_{f^j}  q_{s_{\Gamma(f^j)}}  t_{(f^j)^*} \right) t_{(e^1)^*}  \\
&\qquad \text{by Lemma~\ref{lem-outsplit-uniq-source}} \\
&=  \sum_{ e\in v \Lambda^{\mathbf{e}_i } } t_{e^1} \left( \sum_{ f \in s(e) \Lambda^{\mathbbm{1}} } t_{f^j}   t_{(f^j)^*} \right) t_{(e^1)^*}  \\
& \qquad \text{Lemma~\ref{lem-paths-2}}\\
&= \sum_{ e\in v \Lambda^{\mathbf{e}_i } } t_{e^1} \left( \sum_{ g \in s(e)^j \Gamma^{\mathbbm{1}} } t_{g}   t_{g^*} \right) t_{(e^1)^*}  = \sum_{ e\in v \Lambda^{\mathbf{e}_i } } t_{e^1} q_{ s(e)^j } t_{(e^1)^*} \\
&= \sum_{ e\in v \Lambda^{\mathbf{e}_i } } t_{e^1} q_{ s_\Gamma(e^1) } t_{(e^1)^*}= \sum_{ e\in v \Lambda^{\mathbf{e}_i } } t_{e^1} t_{(e^1)^*} = \sum_{g \in v^1 \Gamma^{\mathbf{e}_i}} t_{g} t_{g^*} \\
&= q_{v^1}= P_v.
\end{align*}

Note that $\psi(p_v)= q_{v^1} \in \mathrm{KP}_\mathbb{F}(\Gamma)_0$, $\psi( s_e ) = \sum_{ f \in s(e)\Lambda^{\mathbbm{1}}} t_{e^1} t_{ f^j } t_{(f^1)^*} \in \mathrm{KP}_\mathbb{F}(\Gamma)_{\mathrm{deg}(e)}$, and $\psi( s_{e^*} ) = \sum_{ f \in s(e)\Lambda^{\mathbbm{1}}}   t_{f^1}t_{ (f^j)^* }t_{(e^1)^*} \in \mathrm{KP}_\mathbb{F}(\Gamma)_{-\mathrm{deg}(e)}$, where $s_\Gamma(e^1)= s(e)^j$.  As $\{p_v, t_e : v \in \Lambda^0, e \in \Lambda^1 \}$ generates $\mathrm{KP}_\mathbb{F}(\Gamma)$, we get $\psi$ is a $\ZZ^k$-graded homomorphism.  Since $\mathbb{F}$ is a field and $\psi( p_v) = q_{v^1} \neq 0$ for all $v \in \Lambda^0$, then $\psi$ is an injection by \cite[Theorem~4.1]{KPAlgebras}.
\end{proof}

\begin{defn}
Let $(\Lambda,d)$ be a $k$-graph. Let $\{q,s\}$ be the universal Kumjian-Pask $\Lambda$-family, and let $\{Q,S\}$ be the universal Cuntz-Krieger $\Lambda$-family. We set the following notation for the diagonals associated with  $\Lambda$:
\begin{enumerate}
    \item the \textit{algebraic diagonal of $\Lambda$}
\[\mathcal{D}_{alg}(\Lambda) := \mathrm{span}_\mathbb{C}\{S_\lambda S_\lambda^* : \lambda \in \Lambda\} \subset C^*(\Lambda)\]
\item the \textit{diagonal of $\Lambda$}
\[ \mathcal{D}(\Lambda) := \overline{\mathcal{D}_{\mathrm{alg}}(\Lambda)} \subset C^*(\Lambda)\]
\item the \textit{diagonal of the Kumjian-Pask algebra}
\[ \mathcal{D}(\mathrm{KP}_\mathbb{F}(\Lambda)) := span_\mathbb{F}\{s_\lambda s_{\lambda^*} : \lambda \in \Lambda\} \subset \mathrm{KP}_\mathbb{F}(\Lambda)\]
\end{enumerate}
\end{defn}

We now compute the image of the injective graded homomorphism constructed in Theorem~\ref{thm-KP} and show it is also diagonal preserving.

\begin{thm}\label{thm-KPA-new}
Let $\Lambda$ be a source-free, degree $B$ sink-free if $k \geq 3$, row-finite $k$-graph, and let $(\Gamma, d_\Gamma)$ be the result of LiMaR-splitting $\Lambda$ in degree $B$. Let $\{p_v,s_e : v \in \Lambda^0, e \in \Lambda^1\}$ and $\{q_{v^i},t_{e^j} : v^i \in \Gamma^0, e^j \in \Gamma^1\}$ denote the universal Kumijan-Pask $\Lambda$- and $\Gamma$-families, respectively.  Let $X = \{p_{v^1} : v \in \Lambda^0\}$, and let $P_X = \sum_{v\in X} q_{v}$.  Then there exists a diagonal-preserving $\mathbb{Z}^k$-graded isomorphism $\psi \colon \mathrm{KP}_{\mathbb{F}} (\Lambda) \to P_X \mathrm{KP}_{\mathbb{F}}(\Gamma)P_X$ such that $\psi( \mathcal{D}(\mathrm{KP}_{\mathbb{F}} (\Lambda) ) =P_X \mathcal{D}(\mathrm{KP}_{\mathbb{F}} (\Gamma))$.   
\end{thm}

\begin{proof} 
By Theorem~\ref{thm-KP}, there exists an injective, graded homomorphism $\psi \colon \mathrm{KP}_{\mathbb{F}} (\Lambda) \to \mathrm{KP}_{\mathbb{F}} (\Gamma)$.  We will show that $\mathrm{Im} (\psi) = P_X \mathrm{KP}_\mathbb{F}(\Gamma)P_X$ and $\psi( \mathcal{D}(\mathrm{KP}_{\mathbb{F}} (\Lambda) ) =P_X \mathcal{D}(\mathrm{KP}_{\mathbb{F}} (\Gamma))$.

First, we prove that $\mathrm{Im}(\psi) \subseteq P_X \mathrm{KP}_\mathbb{F}(\Gamma) P_X$. Fix $w \in \Lambda^0$. Then $\psi(p_{w})=q_{w^1}$. Since $q_{w^1}$ is an idempotent, then $q_{w^1}=q_{w^1}q_{w^1}q_{w^1}$, and since the idempotents in the Kumjian-Pask $\Gamma$-family are orthogonal we have that $q_{w^1}q_{w^1}q_{w^1}=\left(\sum\limits_{v\in X}q_v \right) q_{w^1}\left(\sum\limits_{v\in X}q_v \right) =P_Xq_{w^1}P_X$. Otherwise, for $\lambda \in \Lambda^{\neq \mathbf{0}}$ with $s_\Gamma(\lambda^1) = s(\lambda)^j$ we have by Lemma~\ref{lem-paths} that
\begin{multline*}
    \psi(s_\lambda)
    =\sum\limits_{f\in s(\lambda)\Lambda^{\mathbb{1}}}t_{\lambda^1}t_{f^j}t_{(f^1)^*}
    =\sum\limits_{f\in s(\lambda)\Lambda^{\mathbb{1}}}
q_{r(\lambda^1)}t_{\lambda^1}t_{f^j}t_{(f^1)^*}q_{r(f^1)}\\
    =\sum\limits_{f\in s(\lambda)\Lambda^{\mathbb{1}}}
    \left( \left(\sum\limits_{v\in X}q_v\right) t_{\lambda^1}t_{f^j}t_{(f^1)^*}\left(\sum\limits_{v\in X}q_v \right)\right) \\
    = P_X\left(\sum\limits_{f\in s(\lambda)\Lambda^{\mathbb{1}}}t_{\lambda^1}t_{f^j}t_{(f^1)^*}\right)P_X
    =P_X\psi(s_\lambda)P_X,
\end{multline*} 
where the third equality follows because $q_{v}t_{\lambda^1} = 0$ for any $v \neq r_\Gamma(\lambda^1) = r(\lambda)^1$, and similarly $t_{(f^1)^*}p_v = 0$ for any $v \neq r_\Gamma(f^1) =s(\lambda)^1$. We conclude that $\mathrm{Im}(\psi)\subseteq P_X\mathrm{KP}_\mathbb{F}(\Gamma)P_X$. 

Now we prove $P_X\mathrm{KP}_\mathbb{F}(\Gamma)P_X \subseteq \mathrm{Im}(\psi)$, we first note that $P_X\mathrm{KP}_\mathbb{F}(\Gamma)P_X = \mathrm{span}_\mathbb{F}\{t_{\gamma}t_{\delta^*}: \gamma, \delta \in \Gamma, s_{\Gamma}(\gamma) = s_{\Gamma}(\delta), \text{ and } r_{\Gamma}(\gamma), r_{\Gamma}(\delta) \in X \}$. Fix $t_\gamma t_{\delta^*} \in P_X \mathrm{KP}_\mathbb{F}(\Gamma)P_X$. If both $\gamma, \delta \in \Gamma^0$, then $\gamma = \delta = v^1$ for some $v \in \Lambda$, and $t_\gamma t_{\delta^*} = q_{v^1}q_{v^1} = q_{v^1} = \psi(p_v) \in \mathrm{Im}(\psi)$. If only one of $\gamma$ or $\delta$ belongs to $\Gamma^0$, without loss of generality assume $\gamma \in \Gamma^0$ and $\delta \in \Gamma^{\neq \mathbf{0}}$. Then $\gamma = v^1$ for some $v \in \Lambda$ and so $s_\Gamma(\delta) = s_\Gamma(\gamma) = v^1.$ Moreover, $r_\Gamma(\delta) \in X$ implies that $\delta = \mu^1$ for some $\mu \in \Lambda^{\neq \mathbf{0}}$, and 

\begin{equation*}
    \psi(p_v) \psi(s_{\mu^*}) = q_{v^1} \sum_{f \in v \Lambda^\mathbb{1}}t_{f^1}t_{(f^1)^*}t_{(\mu^1)^*} = q_{v^1}q_{v^1}t_{(\mu^1)^*} = t_\gamma t_{\delta^*}
\end{equation*}
because $p_{v^1}$ is a idempotent and $\mathrm{KP}_\mathbb{F}(\Gamma)$ satisfies \ref{KP4}. Now assume both $\gamma, \delta \in \Gamma^{\neq \mathbf{0}}$. Since $r_\Gamma(\gamma), r_\Gamma(\delta) \in X$, then $\gamma = \lambda^1$ for some $\lambda \in \Lambda^{\neq \mathbf{0}}$ and $\delta = \mu^1$ for some $\mu \in \Lambda^{\neq \mathbf{0}}$. Let $s_\Gamma(\gamma) = s_\Gamma(\delta) = v^j$ for some $v \in \Lambda^0$. Then by Lemma~\ref{lem-paths}, we have

\begin{multline*}
    \psi(s_\lambda) \psi(s_{\mu^*}) = \sum_{f \in v \Lambda^\mathbb{1}} t_{\lambda^1} t_{f^j} t_{(f^1)^*} \sum_{g \in v\Lambda^\mathbb{1}}t_{g^1}t_{(g^j)^*}t_{(\mu^1)^*} \\
    = \sum_{f \in v \Lambda^\mathbb{1}} t_{\lambda^1} t_{f^j} t_{(f^1)^*} t_{f^1}t_{(f^j)^*}t_{(\mu^1)^*} \quad \text{ by \ref{KP3}} \\
    =\sum_{f \in v\Lambda^\mathbb{1}} t_{\lambda^1} t_{f^j} q_{s_\Gamma(f^1)}t_{(f^j)^*}t_{(\mu^1)^*} \quad \text{ by \ref{KP3}} \\
    =\sum_{f \in v \Lambda^\mathbb{1}} t_{\lambda^1} t_{f^j} t_{(f^j)^*}t_{(\mu^1)^*} \quad \text{ by Lemmas \ref{lem-outsplit-uniq-source} and \ref{lem: consequences of CKs}} \\
    =t_{\lambda^1} \left(\sum_{f \in v \Lambda^\mathbb{1}}  t_{f^j} t_{(f^j)^*}\right)t_{(\mu^1)^*}  = t_{\lambda^1} q_{v^j}t_{(\mu^1)^*}  = t_{\lambda^1} t_{(\mu^1)^*} \quad \text{ by \ref{KP4} and Lemma \ref{lem: consequences of CKs}}. 
\end{multline*}

This proves $\{t_{\gamma}t_{\delta^*}: \gamma, \delta \in \Gamma, s_{\Gamma}(\gamma) = s_{\Gamma}(\delta), \text{ and } r_{\Gamma}(\gamma), r_{\Gamma}(\delta) \in X \} \subseteq \mathrm{Im}(\psi)$. Hence, every element in $P_X\mathrm{KP}_\mathbb{F}(\Gamma)P_X = \mathrm{span}\{t_{\gamma}t_{\delta^*}: \gamma, \delta \in \Gamma, s_{\Gamma}(\gamma) = s_{\Gamma}(\delta), \text{ and } r_{\Gamma}(\gamma), r_{\Gamma}(\delta) \in X \}$ is in $\mathrm{Im}(\psi)$. Consequently, $\mathrm{Im} (\psi)=P_X\mathrm{KP}_\mathbb{F}(\Gamma)P_X$.

Lastly, we prove that $\psi( \mathcal{D}(\mathrm{KP}_\mathbb{F}(\Lambda)) ) =P_X \mathcal{D}(\mathrm{KP}_\mathbb{F}(\Gamma))$. By the above computations, we see that $\psi ( s_\mu s_{\mu^*} ) = t_{\mu^1}t_{(\mu^1)^*} \in P_X \mathcal{D}(\mathrm{KP}_\mathbb{F}(\Gamma))$ and if $\gamma \in \Gamma$ and $r_{\Gamma}(\gamma) \in X$, then $\gamma= \mu^1$ and $t_\gamma t_{\gamma^*} = \psi( s_\mu s_{\mu^*})$.  Consequently, $\psi( \mathcal{D}(\mathrm{KP}_\mathbb{F}(\Lambda)) ) =P_X \mathcal{D}(\mathrm{KP}_\mathbb{F}(\Gamma))$.
\end{proof}

\begin{cor}\label{cor:higher-rank-main}
Let $\Lambda$ be a $k$-graph that is degree $B$ sink-free if $k\geq 3$, and let $\Gamma$ be the LiMaR-split of $\Lambda$ at $w$ of degree $B$. Let $X = \{v^1 : v \in \Lambda^0\}$, and let $P_X = \sum_{v\in X} Q_v$, where $\{Q_{v^i}, T_{e^j}: v^i \in \Gamma^0, e^j \in \Gamma^1\}$ is the universal Cuntz-Krieger $\Gamma$-family. Then there exists a $\ast$-isomorphism $\overline{\psi} \colon C^*(\Lambda) \to P_XC^*(\Gamma)P_X$ such that $\psi( \mathcal{D}(\Lambda) ) =P_X \mathcal{D}(\Gamma)$ and 
$$
\beta_z \circ \overline{\psi} = \overline{\psi} \circ \alpha_z
$$
for all $z \in \mathbb{T}^k$.  Here, $\beta_z$ and $\alpha_z$ are the canonical gauge actions on $C^*(\Gamma)$ and $C^*(\Lambda),$ respectively. That is, $\alpha_z(t_e)=z^{d(e)}t_e$ and $\alpha_z(p_v)=p_v$ for $z\in\mathbb{T}^{k},e\in\Lambda^{1}$, and $v\in\Lambda^{0}$, and $\beta_z$ is defined similarly.
\end{cor}

\begin{proof}
By Theorem~\ref{thm-KPA-new} with $\mathbb{F}= \mathbb{C}$, there exists diagonal-preserving $\mathbb{Z}^k$-graded isomorphism $\psi \colon \mathrm{KP}_{\mathbb{C}} (\Lambda) \to P_X \mathrm{KP}_{\mathbb{C}}(\Gamma)P_X$ such that $\psi( \mathcal{D}(\mathrm{KP}_{\mathbb{C}} (\Lambda) ) =P_X \mathcal{D}(\mathrm{KP}_{\mathbb{C}} (\Gamma))$.  

Let $\{p_v,s_e : v \in \Lambda^0, e \in \Lambda^1\}$ and $\{q_{v^i},t_{e^j} : v^i \in \Gamma^0, e^j \in \Gamma^1\}$ denote the universal Kumijan-Pask $\Lambda$- and $\Gamma$-families, respectively and let $\{P_v,S_e : v \in \Lambda^0, e \in \Lambda^1\}$ denote the universal Cuntz-Krieger $\Lambda$-family.  By \cite[Proposition~7.3]{KPAlgebras} $\mathrm{KP}_{\mathbb{C}} (\Lambda)$ embeds as a dense $*$-subalgebra of $C^*(\Lambda)$ via the embedding $\iota_\Lambda$ defined by $\iota_\Lambda(p_v) = P_v$, $\iota_\Lambda(s_e) =S_e$, and $\iota_\Lambda(s_{e^*})=S_e^*$.  It is now clear that this embedding sends $\mathcal{D}(\mathrm{KP}_{\mathbb{C}} (\Lambda) )$ to a dense $*$-subalgebra of $\mathcal{D}(\Lambda)$.  Similarly, $\mathrm{KP}_{\mathbb{C}} (\Gamma)$ is a dense $*$-subalgebra of $C^*(\Lambda)$ via the embedding $\iota_\Gamma$ defined by $\iota_\Gamma(q_{v^i}) =Q_{v^i}$ and $\iota_\Gamma(t_{e^i})= T_{e^i}$, $\iota_{\Gamma}(t_{(e^i)^*}) = T_{e^i}^*$, and $\iota_\Gamma(\mathcal{D}(\mathrm{KP}_{\mathbb{C}} (\Gamma) ))$ is a dense $*$-subalgebra of $\mathcal{D}(\Gamma)$.  Hence, $\psi$ induces a diagonal-preserving $\ast$-isomorphism $\overline{\psi} \colon C^*(\Lambda) \to P_X C^*(\Gamma)P_X$.  

We now show that
$$
\beta_z \circ \overline{\psi} = \overline{\psi} \circ \alpha_z
$$
for all $z \in \mathbb{T}^k$.  First observe that if $\mu, \nu \in \Gamma$ with $d(\mu)-d(\nu)=m$, then 
\[
\beta_z( \iota_\Gamma( \mu \nu^*)) = \beta_z( t_\mu t_\nu^*) = z^{d(\mu)- d(\nu)} t_\mu t_\nu^* = z^m t_\mu t_\nu^*.
\]
Let $v \in \Lambda^0$ and $e \in \Lambda^1$.  Then $p_v \in \mathrm{KP}_{\mathbb{C}}(\Lambda)_0$ and $s_e \in \mathrm{KP}_{\mathbb{C}}(\Lambda)_1$ which implies that $\psi(v) \in \mathrm{KP}_{\mathbb{C}}(\Gamma)_0$ and $\psi(e) \in \mathrm{KP}_{\mathbb{C}}(\Gamma)_1$ as $\psi$ is a $\mathbb{Z}^k$-graded isomorphism.  So, 
\[
\beta_z \circ \overline{\psi}( P_v) = \beta_z \circ  \iota_\Gamma( \psi(p_v)) = \iota_\Gamma( \psi(p_v)) = \overline{\psi}( P_v) =  \overline{\psi}\circ \alpha_z( P_v)
\]
and
\[
\beta_z \circ \overline{\psi}( S_e) = \beta_z \circ  \iota_\Gamma( \psi(s_e)) = z \iota_\Gamma( \psi(s_e)) = z \overline{\psi}( S_e) =  \overline{\psi}( zS_e) =  \overline{\psi}\circ \alpha_z( S_e).
\]
Since $C^*(\Lambda)$ is generated by $\{ P_{v} , S_e : v \in \Lambda^0, e \in \Lambda^1 \}$, we get the desired result.
\end{proof}

For a $\ZZ^k$-graded ring $R$, we give $R \otimes \mathsf{M}_{\infty}$ the following matrix grading: for each $m \in \ZZ^k$, the $m$th component $R \otimes \mathsf{M}_{\infty}$ is $\mathsf{M}_\infty(R_m)$ the set of all infinite matrix with entries $R_m$ with finite support.

\begin{cor}\label{cor:stb}
Let $\Lambda$ be a source-free, degree $B$ sink-free if $k \geq 3$, row-finite $k$-graph, and let $\Gamma$ be the LiMaR-split of $\Lambda$ at $w$ of degree $B$. Then 
\begin{enumerate}
\item \label{stb1} there exists a diagonal-preserving $*$-isomorphism $\Psi \colon C^*(\Lambda) \otimes \Kk \to C^*(\Gamma) \otimes \Kk$ such that $(\beta_z \otimes \operatorname{id}) \circ \Psi = \Psi \circ (\alpha_z \otimes \operatorname{id})$, where \(\alpha, \beta\) are the canonical gauge action of $C^*(\Lambda)$ and $C^*(\Gamma)$ respectively, and 

\item \label{stb2} there is a diagonal-preserving $\mathbb{Z}^k$-graded $*$-isomorphism $\Phi \colon \mathrm{KP}_\mathbb{F}(\Lambda) \otimes \mathsf{M}_\infty \to \mathrm{KP}_\mathbb{F}(\Gamma) \otimes \mathsf{M}_\infty$.
\end{enumerate}

\end{cor}

\begin{proof}
By \cite[Theorem~6.1]{CR2021}, \eqref{stb1} and \eqref{stb2} are equivalent.  So, we prove \eqref{stb1}.  By Corollary~\ref{cor:higher-rank-main} and \cite[Corollary~11.3 (8) $\iff$ (9)]{CRST}, it is enough to prove that the ideal of the fixed point algebra $C^*(\Gamma)^\beta = \{a \in C^*(\Gamma) \mid \beta_z(a) = a \, \forall z \in \mathbb{T}\}$ generated by $P = \sum_{ v \in \Lambda^0 } p_{v^1}$ is $C^*(\Gamma)^\beta$.  To prove this, we prove that $p_x$ is in the ideal of $C^*(\Gamma)^\beta$ generated by $P = \sum_{ v \in \Lambda^0 } p_{v^1}$ for all $x \in \Gamma^{0}$.

Let $v \in \Lambda^0$.  By construction of $\Gamma$, the edges in $\Gamma$ with degree $B$ and range $v^i$ are the edges $e^i$ with $e \in v\Lambda^B$.  Fix $1 \leq i \leq n(v)$ and $e \in v \Lambda^B$, and set $T_{e, i} = s_{e^i} s_{e^1}^*$.  Note that
$\beta_z (T_{e,i}) = z^{B}s_{e^i} \overline{z}^B s_{e^1}^* = T_{e,i}$ and $T_{e,i}P = T_{e,i}$, so $T_{e,i}$ is in the ideal of $C^*(\Gamma)^\beta$ generated by $P$. Then $T_{e,i} T_{e,i}^* = s_{e^i} s_{e^i}^*$ proves that $s_{e^i} s_{e^i}^*$ is also in the ideal of $C^*(\Gamma)^\beta$ generated by $P$.  Therefore, $p_{v^i}$ is in the ideal of $C^*(\Gamma)^\beta$ generated by $P = \sum_{ v \in \Lambda^0 } p_{v^1}$ since 
\[p_{v^i} = \sum_{ e \in v \Lambda^B} s_{e^i} s_{e^i}^*.\]
Completing the proof of \eqref{stb1}.
\end{proof}

For each $k \in \mathbb{N}$, let $\overline{\Omega}_k$ be the $k$-graph $\{ (m,n) \in \ZZ^k \times \ZZ^k : m \leq n\}$ with degree map $d(m,n)=n-m$, and range and source maps $r(m, n) = (m,m)$ and $s(m,n) =(n,n)$. Let $\Lambda$ be a source-free, sink-free, row-finite $k$-graph with finitely many vertices.  Then $\overline{X}_{\Lambda}$ will denote the set consisting of all $k$-graph morphisms from $\overline{\Omega}_k$ to $\Lambda$.  We equip $\overline{X}_{\Lambda}$ with a topology generated by 
\[
\mathcal{Z}( (m,n), \lambda) := \{ x \in \overline{X}_\Lambda : x(m,n) = \lambda \},
\]
with $(m,n) \in \overline{\Omega}_k$ and $\lambda \in \Lambda^{n-m}$.  For $m \in \NN^k$, we denote by $\overline{\sigma}_\Lambda^m$ the homeomorphism from $\overline{X}_{\Lambda}$ to $\overline{X}_{\Lambda}$ by 
\[
\overline{\sigma}_\Lambda^m(x) (p ,q):= x(p+m, q+m)
\]
for $(p,q) \in \overline{\Omega}_k$ and $x \in \overline{X}_{\Lambda}$.

\begin{cor}\label{cor:dynamics}
Let $\Lambda$ be a source-free, sink-free, row-finite $k$-graph with finitely many vertices and let $\Gamma$ be the LiMaR-split of $\Lambda$ at $w$ of degree $B$.  Then there exists a homeomorphism $h \colon \overline{X}_\Lambda \to \overline{X}_\Gamma$ such that $\overline{\sigma}_\Gamma^m \circ h(x) = h \circ \overline{\sigma}_\Lambda^m(x)$ for all $m \in \NN^k$ and for all $x \in \overline{X}_\Lambda$.
\end{cor}

\begin{proof}
This follows from Corollary~\ref{cor:stb} and \cite[Theorem~6.1]{CR2021}.
\end{proof}

\textbf{Acknowledgments}: The authors would like to thank Elizabeth Gillaspy for bringing this research question to our attention and for her continued support.  We thank Ian Oberbillig, Samuel Joseph Lippert, and Sarah Reznikoff for their helpful comments. This work was supported by the St. Olaf Collaborative Undergraduate Research and Inquiry (CURI) Program with funding from the Kay Winger Blair Endowment.  E. Ruiz was partially supported by a Collaborative-NSF-BSF grant (DMS 2452325).

\bibliography{bibliography}
\bibliographystyle{plain}

\end{document}